 or
\documentclass[10pt,twoside]{siamltex}

\usepackage[english]{babel}
\usepackage{graphicx,epstopdf,epsfig}
\usepackage{amsfonts,epsfig,fancyhdr,graphics, hyperref,amsmath,amssymb}

\newcommand{\Names}{Patricia Mariela Morillas}
\newcommand{\Title}{Expressions and characterizations for the Moore--Penrose inverse}

\newtheorem{remark}[theorem]{Remark}
\newtheorem{example}[theorem]{Example}

\newcommand{\norm}[1]{\left\Vert#1\right\Vert}

\newcommand{\set}[1]{\left\{#1\right\}}
\newcommand{\paren}[1]{\left(#1\right)}

\newcommand{\di}{\displaystyle}

\begin{document}

\bibliographystyle{plain}

\setcounter{page}{1}

\thispagestyle{empty}

 \title{Expressions and characterizations for the\\ Moore--Penrose inverse of operators and matrices}

\author{
Patricia Mariela Morillas\thanks{Instituto de Matem\'{a}tica
Aplicada San Luis (UNSL-CONICET), Ej\'{e}rcito de los Andes 950,
5700 San Luis, Argentina (pmmorillas@gmail.com). Supported by Grants
PIP 112-201501-00589-CO (CONICET) and PROIPRO 03-1420 (UNSL).}}

\markboth{\Names}{\Title}

\maketitle

\begin{abstract}
Under certain conditions, we prove that the Moore--Penrose inverse
of a sum of operators is the sum of the Moore--Penrose inverses.
From this, we derive expressions and characterizations for the
Moore--Penrose inverse of an operator that are useful for its
computation. We give formulations of them for finite matrices and
study the Moore--Penrose inverse of circulant matrices and of
distance matrices of certain graphs.
\end{abstract}




\section{Introduction}\label{S Introduction}

A generalization of the concept of inverse for matrices was first
introduced by Moore \cite{Moore (1920), Moore-Barnard (1935)}. Then,
this generalized inverse was independently reintroduced and studied
by Bjerhammer \cite{Bjerhammar (1951)} and Penrose \cite{Penrose
(1955)}. The now commonly called \emph{Moore--Penrose inverse} has
been also defined and studied for operators in Hilbert spaces (see
e.g., \cite{Petryshyn (1967), Ding-Huang (1994)} and the references
therein). It has numerous applications to physics, statistics,
optimization theory, solution of differential and integral
equations, prediction theory, control system analysis, etc. For more
details see, e.g., \cite{Ben-Israel-Greville (2003), Campbell-Meyer
(2009), Wang-Wei-Qiao (2018), Baksalary-Trenkler (2021)}.

Let $\mathcal{H}_{1}$ and $\mathcal{H}_{2}$ be Hilbert spaces over
$\mathbb{F}=\mathbb{R}$ or $\mathbb{F}=\mathbb{C}$. Let
$\mathcal{B}(\mathcal{H}_{1}, \mathcal{H}_{2})$ denotes the set of
bounded linear operators from $\mathcal{H}_{1}$ to
$\mathcal{H}_{2}$. The set of elements in
$\mathcal{B}\!\paren{\mathcal{H}_{1}, \mathcal{H}_{2}}$ with closed
range will be denoted with $\mathcal{BC}\!\paren{\mathcal{H}_{1},
\mathcal{H}_{2}}$. Given $A \in \mathcal{B}(\mathcal{H}_{1},
\mathcal{H}_{2})$, we denote the adjoint, the null space, and the
range of $A$ by $A^{*}$, ${\rm N}(A)$ and ${\rm R}(A)$,
respectively. If $A \in \mathcal{BC}\!\paren{\mathcal{H}_{1},
\mathcal{H}_{2}}$, then $\mathcal{H}_{1}={\rm N}\!\paren{A} \oplus
{\rm R}\!\paren{A^{*}}$ and $\mathcal{H}_{2}={\rm R}\!\paren{A}
\oplus {\rm N}\!\paren{A^{*}}$, where $\oplus$ denotes orthogonal
sum. If $\mathcal{W}$ is a closed subspace of a Hilbert space, then
$P_{\mathcal{W}}$ denotes the orthogonal projection onto
$\mathcal{W}$.

\begin{definition}\label{D MPI}
{\rm Let $A \in \mathcal{BC}\!\paren{\mathcal{H}_{1},
\mathcal{H}_{2}}$. The unique solution $X \in
\mathcal{BC}\!\paren{\mathcal{H}_{2}, \mathcal{H}_{1}}$ of the
system of operator equations
\[(1)\, AXA = A, \,\,(2)\, XAX = X, \,\,(3)\,
\paren{AX}^{*} = AX \text{ and } \,\,(4)\,\paren{XA}^{*} =
  XA,\]
is called the \emph{Moore--Penrose inverse} of $A$ and is denoted by
$A^{\dag}$. Any $X \in \mathcal{B}\!\paren{\mathcal{H}_{2},
\mathcal{H}_{1}}$ that satisfies (1), (3) and (4) is known as a
\emph{$\{1, 3, 4\}$-inverse} of $A$.}
\end{definition}
One of the most important properties used in the applications of the
Moore--Penrose inverse is that the \emph{minimal norm least squares
problem}

\centerline{$\di{\min\norm{x}}$ subject to $\di{\norm{Ax-g}=\min_{f
\in \mathcal{H}_{1}}\norm{Af-g}}$,}

\noindent has the unique solution $x=A^{\dag}g$.

We note that if $A \in \mathcal{B}(\mathcal{H}_{1},
\mathcal{H}_{2})$, then $A^{\dag}$ exists if and only if ${\rm
R}\!\paren{A}$ is closed, and ${\rm R}\!\paren{A}$ is closed if and
only if ${\rm R}\!\paren{A^{*}}$ is closed. Let $A \in
\mathcal{BC}\!\paren{\mathcal{H}_{1}, \mathcal{H}_{2}}$. Then
$\paren{A^{\dag}}^{\dag}=A$ and
$\paren{A^{*}}^{\dag}=\paren{A^{\dag}}^{*}$. If $A$ is invertible,
then $A^{\dag}=A^{-1}$. The Moore--Penrose inverse is characterized
in the following theorem (see \cite[Theorem]{Petryshyn (1967)} or
\cite[Theorem 9.3]{Ben-Israel-Greville (2003)}):

\begin{theorem}\label{T MPI caracterizacion}
If $A \in \mathcal{BC}\!\paren{\mathcal{H}_{1}, \mathcal{H}_{2}}$,
then $A^{\dag} \in \mathcal{BC}\!\paren{\mathcal{H}_{2},
\mathcal{H}_{1}}$ is the unique solution of anyone of the following
equivalent systems:
\begin{enumerate}
  \item[(i)] $AX = P_{{\rm R}\!\paren{A}}$, ${\rm N}\!\paren{X^{*}} = {\rm
  N}\!\paren{A}$.
  \item[(ii)] $AX = P_{{\rm R}\!\paren{A}}$, $XA = P_{{\rm R}\!\paren{A^{*}}}$, $XAX =
  X$.
  \item[(iii)] $XAA^{*} = A^{*}$, $XX ^{*} A^{*} = X$.
  \item[(iv)] $XAP_{{\rm R}\!\paren{A^{*}}} = P_{{\rm
  R}\!\paren{A^{*}}}$, $XP_{{\rm N}\!\paren{A^{*}}} = 0$.
  \item[(v)] $XA = P_{{\rm R}\!\paren{A^{*}}}$, ${\rm N}\!\paren{X} = {\rm
  N}\!\paren{A^{*}}$.
  \item[(vi)] $AX = P_{{\rm R}\!\paren{A}}$, $XA = P_{{\rm R}\!\paren{X}}$.
\end{enumerate}
\end{theorem}

\subsection{Our contributions and the organization of the paper}

In Section~\ref{S MPI of sums of operators}, we prove that under
certain conditions, the Moore--Penrose inverse of a sum of operators
is the sum of the Moore--Penrose inverses (Theorem~\ref{T MPI
expression sum}). From this, we give a result for $\{1, 3,
4\}$-inverses (Corollary~\ref{C 134 inverse}) and derive expressions
and characterizations for the Moore--Penrose inverse that are useful
for its computation (Theorem~\ref{T MPI expression sum injective},
Corollary~\ref{C MPI expression sum injective}, Theorem~\ref{T MPI
expression sum surjective} and Corollary~\ref{C MPI expression sum
surjective}). Part of them can be viewed as extensions of the
following well-known result:
\begin{proposition}\label{P MPI usando adjunta}
Let $A \in \mathcal{BC}\!\paren{\mathcal{H}_{1}, \mathcal{H}_{2}}$.
Then:
\begin{enumerate}
  \item[(i)] $A^{\dag}=\paren{A^{*}A}^{\dag}A^{*}=A^{*}\paren{AA^{*}}^{\dag}$.
  \item[(ii)] If ${\rm N}\!\paren{A}=\set{0}$, then $A^{*}A$ is invertible and
$A^{\dag}=\paren{A^{*}A}^{-1}A^{*}$.
  \item[(iii)] If ${\rm R}\!\paren{A}=\mathcal{H}_{2}$,
then $AA^{*}$ is invertible and $A^{\dag}=A^{*}\paren{AA^{*}}^{-1}$.
\end{enumerate}
\end{proposition}
In particular, we note that in Definition~\ref{D MPI} and in
Theorem~\ref{T MPI caracterizacion}, $A^{\dag}$ is characterized as
the solution of systems of equations, or of operators equations
where the solution must satisfy some restrictions on its null space
or on its range space. The importance of Proposition~\ref{P MPI
usando adjunta}(ii)(iii) from the computational point of view is
that if $A \in \mathcal{BC}\!\paren{\mathcal{H}_{1},
\mathcal{H}_{2}}$ is injective (surjective), then $A^{\dag}$ is the
unique solution of the single equation $\paren{A^{*}A}X=A^{*}$
(resp. $X\paren{AA^{*}}=A^{*}$). Our results permit to address the
computation of $A^{\dag}$ in a similar manner as a solution of a
single equation, in cases in which $A$ is not necessarily injective
or surjective.

Theorem~\ref{T MPI expression sum}(iv) extends \cite[Lemma
1.7]{Penrose (1955)} about the Moore--Penrose inverse of sums of
matrices to sums of operators with closed ranges. In Section~\ref{S
Formulations for matrices}, we consider other generalizations of
\cite[Lemma 1.7]{Penrose (1955)} given in \cite{Cline (1965),
Hung-Markham (1977), Fill-Fishkind (1999)} for sums of two matrices
and in \cite{Arias-Corach-Maestripieri (1999)} for sums of two
operators. We also mention other approaches to the study of the
Moore--Penrose inverse of sums of matrices that appear in
\cite{Sivakumar (2020), Baksalary-Sivakumar-Trenkler (2022), Tian
(1998), Tian (2001), Tian (2005)}. Then, we give formulations of
results of Section~\ref{S MPI of sums of operators} for finite
matrices (Theorem~\ref{T MPI expression}, Corollary~\ref{C MPI
expression full rank} and Theorem~\ref{T MPI expression full rank
comp1}). These formulations provide methods for the obtention of the
Moore--Penrose inverse in the finite-dimensional case. We show that
results appeared in \cite{Ben-Israel-Greville (2003), Bell (1981),
Plonka-Hoffmanny-Weickertz (2016)}, which were proved with different
approaches, are particular cases of Corollary~\ref{C MPI expression
full rank}(iii) (see Remark~\ref{R results literature}). We also
show the relation of Theorem~\ref{T MPI expression sum} with the
singular value decomposition (Remarks~\ref{R Theorem 3.2 Sivakumar
(2020) conditions not necessary} and \ref{R Theorem T MPI expression
sum conditions not necessary}).

In Sections~\ref{S MPI of circulant matrices} and \ref{S MPI
distance matrix graphs}, we use our previous results to determine
closed-form expressions for the entries of the Moore--Penrose
inverse of circulant matrices and of distance matrices of certain
graphs.

Circulant matrices arise in various areas of applied mathematics and
science, such as statistics, physics, signal processing, and coding
theory, among many others (see, e.g., \cite{Davis (1994), Aldrovandi
(2001), Gray (2001)}). Beside of its theoretical interest, having an
explicit expression for the (Moore--Penrose) inverses of circulant
matrices can reduce the computational cost in dealing with them in
some applications. There are several papers that give expressions
for the (Moore--Penrose) inverse of circulant matrices, see e.g.
\cite{Searle (1979), Bell (1981), Fuyong (2011),
Carmona-Encinas-Gago-Jimenez-Mitjana (2015), Bustomi-Barra (2017),
Carmona-Encinas-Jimenez-Mitjana (2021),
Carmona-Encinas-Jimenez-Mitjana (2022)}. In Section~\ref{S MPI of
circulant matrices}, we consider this type of matrices finding the
explicit expressions of the (Moore--Penrose) inverse of some
families of complex circulant matrices (Lemma~\ref{L circ 2,0,1,
...,1}, Proposition~\ref{P MPI circulant dos}, Example~\ref{Ex MPI
circulant mas menos} and Proposition~\ref{P MPI circulant abbb}).
Based on spectral properties of circulant matrices, Theorem~\ref{T
MPI expression sum circulant} provides a way to obtain the
(Moore--Penrose) inverse of circulant matrices in terms of the
(Moore--Penrose) inverses of other circulant matrices. As a
consequence, we get Proposition~\ref{P MPI circulant suma cero} that
shows that we can easily have an explicit expression for the
Moore--Penrose inverse of a circulant matrix for the case in which
the generating vector has components with nonzero-sum from the case
in which the generating vector has components with zero-sum, and
vice versa. Hence, there is no need to separately consider the two
cases. To our best knowledge, this fact was not noted so far in the
literature. It will simplify the study and the application of the
Moore--Penrose inverse of circulant matrices in future works.

Distance matrices of connected graphs have several interesting
properties and have applications in, e.g., chemistry, biology, and
data communication \cite{Bapat (2018), Fiedler (2011)}. In
\cite{Graham-Lovasz (1978)}, Graham and Lov\'{a}sz express the
inverse of the distance matrix of an unweighted tree in terms of the
Laplacian matrix. Since then, there was interest in give expressions
for the (Moore--Penrose) inverse of distance matrices of graphs
using the Laplacian matrix of the graph or a generalization of it
(see e.g. \cite{Bapat-Kirkland-Neumann (2005), Kurata-Bapat (2016),
Balaji-Bapat-Goel (2020), Balaji-Bapat-Goel (2021)}). Section~\ref{S
MPI distance matrix graphs} is devoted to find explicit expression
of the Moore--Penrose inverse of distance matrices $D$ of weighted
trees and of wheel graphs with an odd number of vertices. Our
expressions for $D^{\dag}$ do not involve the Laplacian matrix or a
generalization of it but an invertible matrix which is a $\{1, 3,
4\}$-inverse of $D$ constructed directly from $D$ and its null space
(see (\ref{E MPI D wtree 1}), (\ref{E MPI D wtree 2}) and
Theorem~\ref{T MPI dmwgo}). These expressions give alternatives to
the ones presented in \cite{Kurata-Bapat (2016), Balaji-Bapat-Goel
(2020)} for the computation of $D^{\dag}$.

In Section~\ref{SubS MPI distance matrix weighted tree}, we consider
distance matrices $D$ of weighted trees with all the weights being
nonzero and with sum equal to zero. We note that the expression of
$D^{\dag}$ given in \cite[Theorem 11]{Kurata-Bapat (2016)} can be
viewed as an extension of \cite[Theorems 3 and 4]{Kurata-Bapat
(2015)} for Euclidean distance matrices and the well-known formula
due to Graham and Lov\'{a}sz \cite{Graham-Lovasz (1978)}. In the
expression for $D^{\dag}$ provided in \cite[Theorem 11]{Kurata-Bapat
(2016)} appears a vector that depends on $D^{\dag}$. We show how our
expressions (\ref{E MPI D wtree 1}) or (\ref{E MPI D wtree 2}) can
be used to compute this vector without using $D^{\dag}$. Moreover,
for certain types of weighted trees, we give an explicit expression
of this vector using only the Laplacian matrix and the degree vector
of the tree (Proposition~\ref{P MPI D wtree u Ldelta}).

In Section~\ref{SubS MPI dmwgo}, we deal with the Moore--Penrose
inverse of distance matrices $D$ of wheel graphs with an odd number
of vertices. In \cite{Balaji-Bapat-Goel (2020)}, an explicit
expression of $D^{\dag}$ is given in terms of a generalized
Laplacian matrix extending in this way the classical result of
Graham and Lov\'{a}sz \cite{Graham-Lovasz (1978)}. Here, we give a
closed-form expression for each entry of the inverse of a $\{1, 3,
4\}$-inverse of $D$, and thus of $D^{\dag}$. The principal result is
Theorem~\ref{T MPI dmwgo}. Lemmas~\ref{L MPI dmwgo z1} and \ref{L
MPI dmwgo z2}, used to prove Theorem~\ref{T MPI dmwgo}, are about
these entries and describe properties of them. In Proposition~\ref{P
MPI dmwgo properties 123D}, we give some properties of the inverse
of the $\{1, 3, 4\}$-inverse of $D$ which is used in our expression
of $D^{\dag}$ and of the generalized Laplacian matrix introduced in
\cite{Balaji-Bapat-Goel (2020)}.

\section{Moore--Penrose inverses of sums of operators}\label{S MPI of sums of operators}

The following theorem gives conditions under which the
Moore--Penrose inverse of a sum is the sum of the Moore--Penrose
inverses.

\begin{theorem}\label{T MPI expression sum}
Let $\set{A_{k}}_{k=1}^{K} \subset
\mathcal{BC}\!\paren{\mathcal{H}_{1}, \mathcal{H}_{2}}$. Assume that
${\rm R}\!\paren{A_{k}} \subseteq {\rm N}\!\paren{A_{k'}^{*}}$ and
${\rm R}\!\paren{A_{k}^{*}} \subseteq {\rm N}\!\paren{A_{k'}}$ for
each $k, k' = 1, \ldots, K$, $k \neq k'$. Then:
\begin{enumerate}
  \item[(i)] $A_{k}A_{k'}^{\dag}=0$ and $A_{k}^{\dag}A_{k'}=0$ for each  $k, k' = 1,
\ldots, K$, $k \neq k'$.
  \item[(ii)] ${\rm R}\!\paren{\sum_{k=1}^{K}A_{k}} = \bigoplus_{k=1}^{K}{\rm R}\!\paren{A_{k}}$ and ${\rm R}\!\paren{\sum_{k=1}^{K}A_{k}^{*}}
= \bigoplus_{k=1}^{K}{\rm R}\!\paren{A_{k}^{*}}$.
  \item[(iii)] ${\rm N}\!\paren{\sum_{k=1}^{K}A_{k}} = \bigcap_{k=1}^{K} {\rm N}\!\paren{A_{k}}$ and ${\rm N}\!\paren{\sum_{k=1}^{K}A_{k}^{*}} = \bigcap_{k=1}^{K} {\rm N}\!\paren{A_{k}^{*}}$.
  \item[(iv)] $\paren{\sum_{k=1}^{K}A_{k}}^{\dag}$ exists and $\paren{\sum_{k=1}^{K}A_{k}}^{\dag}=\sum_{k=1}^{K}A_{k}^{\dag}$.
\end{enumerate}
\end{theorem}

\begin{proof}
Part (i) follows from the inclusions ${\rm R}\!\paren{A_{k'}^{\dag}}
= {\rm R}\!\paren{A_{k'}^{*}} \subseteq {\rm N}\!\paren{A_{k}}$ and
${\rm R}\!\paren{A_{k'}} \subseteq {\rm N}\!\paren{A_{k}^{*}} = {\rm
N}\!\paren{A_{k}^{\dag}}$ for each $k, k' = 1, \ldots, K$, $k \neq
k'$.

Since ${\rm R}\!\paren{A_{k}} \subseteq {\rm N}\!\paren{A_{k'}^{*}}$
for each $k, k' = 1, \ldots, K$, $k \neq k'$, we get ${\rm
R}\!\paren{A_{k}} \perp {\rm R}\!\paren{A_{k'}}$ for each $k, k' =
1, \ldots, K$, $k \neq k'$. Then, we have the orthogonal sum
$\bigoplus_{k=1}^{K}{\rm R}\!\paren{A_{k}}$. Clearly, ${\rm
R}\!\paren{\sum_{k=1}^{K}A_{k}} \subseteq \bigoplus_{k=1}^{K}{\rm
R}\!\paren{A_{k}}$ and $\bigcap_{k=1}^{K} {\rm N}\!\paren{A_{k}}
\subseteq {\rm N}\!\paren{\sum_{k=1}^{K}A_{k}}$.

Let $g \in \bigoplus_{k=1}^{K}{\rm R}\!\paren{A_{k}}$. Then there
exist $f_{1}, \ldots, f_{K} \in \mathcal{H}_{1}$ such that
$g=\sum_{k=1}^{K} A_{k}f_{k}=\sum_{k=1}^{K} A_{k}P_{{\rm
R}\!\paren{A_{k}^{*}}}f_{k}$. Let $f=\sum_{k=1}^{K}P_{{\rm
R}\!\paren{A_{k}^{*}}}f_{k}$. Using that ${\rm R}\!\paren{A_{k}^{*}}
\subseteq {\rm N}\!\paren{A_{k'}}$ for each $k, k' = 1, \ldots, K$,
$k \neq k'$, we get
\[\paren{\sum_{k=1}^{K}A_{k}}f=\paren{\sum_{k=1}^{K}A_{k}}\paren{\sum_{k=1}^{K}P_{{\rm
R}\!\paren{A_{k}^{*}}}f_{k}}=\sum_{k=1}^{K} A_{k}P_{{\rm
R}\!\paren{A_{k}^{*}}}f_{k}=g.\]

\noindent Then $g \in {\rm R}\!\paren{\sum_{k=1}^{K}A_{k}}$. This
sows that $\bigoplus_{k=1}^{K}{\rm R}\!\paren{A_{k}} \subseteq {\rm
R}\!\paren{\sum_{k=1}^{K}A_{k}}$. Hence, ${\rm
R}\!\paren{\sum_{k=1}^{K}A_{k}} = \bigoplus_{k=1}^{K}{\rm
R}\!\paren{A_{k}}$. The equality ${\rm
R}\!\paren{\sum_{k=1}^{K}A_{k}^{*}} = \bigoplus_{k=1}^{K}{\rm
R}\!\paren{A_{k}^{*}}$ can be proved similarly. Therefore, (ii)
holds.

Let $k_{0} \in \set{1, \ldots, K}$. We have $f \in {\rm
N}\!\paren{\sum_{k=1}^{K}A_{k}}$ if and only if
\[A_{k_{0}}f=-\paren{\sum_{k=1, k \neq k_{0}}^{K}A_{k}}f \in {\rm
R}\!\paren{A_{k_{0}}} \cap {\rm R}\!\paren{\sum_{k=1, k \neq
k_{0}}^{K}A_{k}} \subseteq {\rm R}\!\paren{A_{k_{0}}} \cap {\rm
N}\!\paren{A_{k_{0}}^{*}} = \set{0}.\]

\noindent Thus, $f \in {\rm N}\!\paren{A_{k_{0}}}$ for each $k_{0}
\in \set{1, \ldots, K}$. Therefore, $\bigcap_{k=1}^{K} {\rm
N}\!\paren{A_{k}} = {\rm N}\!\paren{\sum_{k=1}^{K}A_{k}}$. In a
similar way, we prove that ${\rm N}\!\paren{\sum_{k=1}^{K}A_{k}^{*}}
= \bigcap_{k=1}^{K} {\rm N}\!\paren{A_{k}^{*}}$. This shows (iii).

By (ii), ${\rm R}\!\paren{\sum_{k=1}^{K}A_{k}}$ is closed. Thus,
$\paren{\sum_{k=1}^{K}A_{k}}^{\dag}$ exists.  Using part now (i),
the rest of (iv) follows from Definition~\ref{D MPI}.
\end{proof}

From the definition of $\{1, 3, 4\}$-inverse and Theorem~\ref{T MPI
expression sum}(i)(iv) we obtain the following result:

\begin{corollary}\label{C 134 inverse}
Let $\set{A_{k}}_{k=1}^{K} \subset
\mathcal{BC}\!\paren{\mathcal{H}_{1}, \mathcal{H}_{2}}$. Assume that
${\rm R}\!\paren{A_{k}} \subseteq {\rm N}\!\paren{A_{k'}^{*}}$ and
${\rm R}\!\paren{A_{k}^{*}} \subseteq {\rm N}\!\paren{A_{k'}}$ for
each $k, k' = 1, \ldots, K$, $k \neq k'$.  Then
$\paren{\sum_{k=1}^{K}A_{k}}^{\dag}$ is a $\{1, 3, 4\}$-inverse of
$A_{k_{0}}$ for each $k_{0} = 1, \ldots, K$.
\end{corollary}

From Theorem~\ref{T MPI expression sum}, we now derive expressions
and characterizations for the Moore--Penrose inverse of an operator.
They are useful for its computation and part of them can be viewed
as an extension of Proposition~\ref{P MPI usando adjunta} to a sum
of operators with closed range. Before we enunciate the results, we
note that, under the hypothesis of Theorem~\ref{T MPI expression
sum}, we always have ${\rm R}\!\paren{A_{k_{0}}^{*}} \subseteq {\rm
N}\!\paren{\sum_{k=1,k\neq k_{0}}^{K}A_{k}}$ and ${\rm
R}\!\paren{\sum_{k=1,k\neq k_{0}}^{K}A_{k}} \subseteq {\rm
N}\!\paren{A_{k_{0}}^{*}}$ for each $k_{0} \in \set{1, \ldots, K}$.

\begin{theorem}\label{T MPI expression sum injective}
Let $\set{A_{k}}_{k=1}^{K} \subset
\mathcal{BC}\!\paren{\mathcal{H}_{1}, \mathcal{H}_{2}}$. Assume that
${\rm R}\!\paren{A_{k}} \subseteq {\rm N}\!\paren{A_{k'}^{*}}$ and
${\rm R}\!\paren{A_{k}^{*}} \subseteq {\rm N}\!\paren{A_{k'}}$ for
each $k, k' = 1, 2, \ldots$, $k \neq k'$. Let $k_{0} \in \set{1,
\ldots, K}$. The following assertions hold:
\begin{enumerate}
  \item[(i)] $A_{k_{0}}^{\dag}=\paren{\sum_{k=1}^{K}A_{k}^{*}A_{k}}^{\dag}\paren{\sum_{k=1}^{K}A_{k}^{*}}-\sum_{k=1,k\neq
k_{0}}^{K}A_{k}^{\dag}$.
  \item[(ii)] $\paren{\sum_{k=1}^{K}A_{k}^{*}A_{k}}A_{k_{0}}^{\dag}=A_{k_{0}}^{*}$ and $A_{k_{0}}^{\dag}=\paren{\sum_{k=1}^{K}A_{k}^{*}A_{k}}^{\dag}A_{k_{0}}^{*}$.
  \item[(iii)] ${\rm N}\!\paren{\sum_{k=1}^{K}A_{k}}=\set{0}$ if and only if
${\rm R}\!\paren{A_{k_{0}}^{*}} = {\rm N}\!\paren{\sum_{k=1,k\neq
k_{0}}^{K}A_{k}}$.
  \item[(iv)] If ${\rm
N}\!\paren{\sum_{k=1}^{K}A_{k}}=\set{0}$, then
$\sum_{k=1}^{K}A_{k}^{*}A_{k}$ is invertible.
\end{enumerate}
\end{theorem}

\begin{proof}
(i): From Theorem~\ref{T MPI expression sum}(iv) and
Proposition~\ref{P MPI usando adjunta}(i), we obtain
\begin{align*}
A_{k_{0}}^{\dag}&=\paren{\sum_{k=1}^{K}A_{k}}^{\dag}-\paren{\sum_{k=1,k\neq
k_{0}}^{K}A_{k}}^{\dag}\\
&=\paren{\paren{\sum_{k=1}^{K}A_{k}^{*}}\paren{\sum_{k=1}^{K}A_{k}}}^{\dag}\paren{\sum_{k=1}^{K}A_{k}^{*}}-\sum_{k=1,k\neq
k_{0}}^{K}A_{k}^{\dag}\\
&=\paren{\sum_{k=1}^{K}A_{k}^{*}A_{k}}^{\dag}\paren{\sum_{k=1}^{K}A_{k}^{*}}-\sum_{k=1,k\neq
k_{0}}^{K}A_{k}^{\dag}.
\end{align*}
(ii): Using Theorem~\ref{T MPI expression sum}(i) we get,
\[\paren{\sum_{k=1}^{K}A_{k}^{*}A_{k}}A_{k_{0}}^{\dag}=A_{k_{0}}^{*}A_{k_{0}}A_{k_{0}}^{\dag}=A_{k_{0}}^{*}P_{{\rm
R}\paren{A_{k_{0}}}}=A_{k_{0}}^{*}.\] From this equality,
\begin{equation}\label{E T MPI expression series injective}
\paren{\sum_{k=1}^{K}A_{k}^{*}A_{k}}^{\dag}\paren{\sum_{k=1}^{K}A_{k}^{*}A_{k}}A_{k_{0}}^{\dag}=\paren{\sum_{k=1}^{K}A_{k}^{*}A_{k}}^{\dag}A_{k_{0}}^{*}.
\end{equation} Since
\[\paren{\sum_{k=1}^{K}A_{k}^{*}A_{k}}^{\dag}\paren{\sum_{k=1}^{K}A_{k}^{*}A_{k}}=P_{{\rm
R}\paren{\sum_{k=1}^{K}A_{k}^{*}A_{k}}}\] and, by Theorem~\ref{T MPI
expression sum}(ii),
\[{\rm
R}\!\paren{\sum_{k=1}^{K}A_{k}^{*}A_{k}}={\rm
R}\!\paren{\paren{\sum_{k=1}^{K}A_{k}^{*}}\paren{\sum_{k=1}^{K}A_{k}}}={\rm
R}\!\paren{\sum_{k=1}^{K}A_{k}^{*}} = \bigoplus_{k=1}^{K}{\rm
R}\!\paren{A_{k}^{*}} \supseteq {\rm R}\paren{A_{k_{0}}^{*}}={\rm
R}\paren{A_{k_{0}}^{\dag}},\] equality (\ref{E T MPI expression
series injective}) becomes
\[A_{k_{0}}^{\dag}=\paren{\sum_{k=1}^{K}A_{k}^{*}A_{k}}^{\dag}A_{k_{0}}^{*}.\]
(iii): Assume that ${\rm N}\!\paren{\sum_{k=1}^{K}A_{k}}=\set{0}$.
Let $f \in {\rm N}\!\paren{\sum_{k=1,k\neq k_{0}}^{K}A_{k}}$. Since
$f = P_{{\rm R}\!\paren{A_{k_{0}}^{*}}}f+P_{{\rm
N}\!\paren{A_{k_{0}}}}f$ and ${\rm R}\!\paren{A_{k_{0}}^{*}}
\subseteq {\rm N}\!\paren{A_{k}}$ for $k \neq k_{0}$,
\[\sum_{k=1}^{K}A_{k}P_{{\rm N}\!\paren{A_{k_{0}}}}f=\sum_{k=1,k\neq
k_{0}}^{K}A_{k}P_{{\rm N}\!\paren{A_{k_{0}}}}f=\sum_{k=1,k\neq
k_{0}}^{K}A_{k}f=0.\] \noindent Consequently, $P_{{\rm
N}\!\paren{A_{k_{0}}}}f=0$ and $f \in {\rm
R}\!\paren{A_{k_{0}}^{*}}$. This shows that ${\rm
N}\!\paren{\sum_{k=1,k\neq k_{0}}^{K}A_{k}}={\rm
R}\!\paren{A_{k_{0}}^{*}}$.

Assume now that ${\rm R}\!\paren{A_{k_{0}}^{*}} = {\rm
N}\!\paren{\sum_{k=1,k\neq k_{0}}^{K}A_{k}}$. By Theorem~\ref{T MPI
expression sum}(iii),
\[{\rm N}\!\paren{\sum_{k=1}^{K}A_{k}} = \bigcap_{k=1}^{K} {\rm
N}\!\paren{A_{k}} = {\rm N}\!\paren{A_{k_{0}}} \cap {\rm
N}\!\paren{\sum_{k=1,k\neq k_{0}}^{K}A_{k}} = {\rm
N}\!\paren{A_{k_{0}}} \cap {\rm R}\!\paren{A_{k_{0}}^{*}} =
\set{0}.\] (iv): It follows from Proposition~\ref{P MPI usando
adjunta}(ii).
\end{proof}
For future references, we enunciate the following straightforward
corollary.
\begin{corollary}\label{C MPI expression sum injective}
Let $\set{A_{k}}_{k=1}^{K} \subset
\mathcal{BC}\!\paren{\mathcal{H}_{1}, \mathcal{H}_{2}}$. Assume that
${\rm R}\!\paren{A_{k}} \subseteq {\rm N}\!\paren{A_{k'}^{*}}$ and
${\rm R}\!\paren{A_{k}^{*}} \subseteq {\rm N}\!\paren{A_{k'}}$ for
each $k, k' = 1, 2, \ldots$, $k \neq k'$. Let $k_{0} \in \set{1,
\ldots, K}$. If $\sum_{k=1}^{K}A_{k}$ is injective, then
$\sum_{k=1}^{K}A_{k}^{*}A_{k}$ is invertible and $A_{k_{0}}^{\dag}$
is the unique solution of the operator equation
$\paren{\sum_{k=1}^{K}A_{k}^{*}A_{k}}X=A_{k_{0}}^{*}$.
\end{corollary}

The following theorem can be proved in a similar manner as was
proved Theorem~\ref{T MPI expression sum injective} or can be proved
applying Theorem~\ref{T MPI expression sum injective} to
$\set{A_{k}^{*}}_{k=1}^{K}$ and then taking adjoints.

\begin{theorem}\label{T MPI expression sum surjective}
Let $\set{A_{k}}_{k=1}^{K} \subset
\mathcal{BC}\!\paren{\mathcal{H}_{1}, \mathcal{H}_{2}}$. Assume that
${\rm R}\!\paren{A_{k}} \subseteq {\rm N}\!\paren{A_{k'}^{*}}$ and
${\rm R}\!\paren{A_{k}^{*}} \subseteq {\rm N}\!\paren{A_{k'}}$ for
each $k, k' = 1, 2, \ldots$, $k \neq k'$. Let $k_{0} \in \set{1,
\ldots, K}$. The following assertions hold:
\begin{enumerate}
  \item[(i)] $A_{k_{0}}^{\dag}=\paren{\sum_{k=1}^{K}A_{k}^{*}}\paren{\sum_{k=1}^{K}A_{k}A_{k}^{*}}^{\dag}-\sum_{k=1,k\neq
k_{0}}^{K}A_{k}^{\dag}$.
  \item[(ii)] $A_{k_{0}}^{\dag}\paren{\sum_{k=1}^{K}A_{k}A_{k}^{*}}=A_{k_{0}}^{*}$ and $A_{k_{0}}^{\dag}=A_{k_{0}}^{*}\paren{\sum_{k=1}^{K}A_{k}A_{k}^{*}}^{\dag}$.
  \item[(iii)] ${\rm R}\!\paren{\sum_{k=1}^{K}A_{k}}=\mathcal{H}_{2}$ if and only if ${\rm R}\!\paren{\sum_{k=1,k\neq k_{0}}^{K}A_{k}} = {\rm
N}\!\paren{A_{k_{0}}^{*}}$.
  \item[(iv)] If ${\rm R}\!\paren{\sum_{k=1}^{K}A_{k}}=\mathcal{H}_{2}$, then
$\sum_{k=1}^{K}A_{k}A_{k}^{*}$ is invertible.
\end{enumerate}
\end{theorem}
We have the following immediate and useful consequence.
\begin{corollary}\label{C MPI expression sum surjective}
Let $\set{A_{k}}_{k=1}^{K} \subset
\mathcal{BC}\!\paren{\mathcal{H}_{1}, \mathcal{H}_{2}}$. Assume that
${\rm R}\!\paren{A_{k}} \subseteq {\rm N}\!\paren{A_{k'}^{*}}$ and
${\rm R}\!\paren{A_{k}^{*}} \subseteq {\rm N}\!\paren{A_{k'}}$ for
each $k, k' = 1, 2, \ldots$, $k \neq k'$. Let $k_{0} \in \set{1,
\ldots, K}$. If $\sum_{k=1}^{K}A_{k}$ is surjective, then
$\sum_{k=1}^{K}A_{k}A_{k}^{*}$ is invertible and $A_{k_{0}}^{\dag}$
is the unique solution of the operator equation
$X\paren{\sum_{k=1}^{K}A_{k}A_{k}^{*}}=A_{k_{0}}^{*}$.
\end{corollary}

The next theorem address the case $\sum_{k=1}^{K}A_{k}$ invertible.

\begin{theorem}\label{T MPI expression sum invertible}
Let $\set{A_{k}}_{k=1}^{K} \subset
\mathcal{BC}\!\paren{\mathcal{H}_{1}, \mathcal{H}_{2}}$. Assume that
${\rm R}\!\paren{A_{k}} \subseteq {\rm N}\!\paren{A_{k'}^{*}}$ and
${\rm R}\!\paren{A_{k}^{*}} \subseteq {\rm N}\!\paren{A_{k'}}$ for
each $k, k' = 1, 2, \ldots$, $k \neq k'$. Let $k_{0} \in \set{1,
\ldots, K}$. If $\sum_{k=1}^{K}A_{k}$ is invertible, then
$A_{k_{0}}^{\dag}$ is the solution of any of the equations
$\paren{\sum_{k=1}^{K}A_{k}}X=P_{{\rm N}\paren{\sum_{k=1,k\neq
k_{0}}^{K}A_{k}^{*}}}$, $X\paren{\sum_{k=1}^{K}A_{k}}=P_{{\rm
N}\!\paren{\sum_{k=1,k\neq k_{0}}^{K}A_{k}}}$.
\end{theorem}
\begin{proof}
Assume that $\sum_{k=1}^{K}A_{k}$ is invertible. From Theorem~\ref{T
MPI expression sum}(iv),
\[A_{k_{0}}^{\dag}=\paren{\sum_{k=1}^{K}A_{k}}^{\dag}-\sum_{k=1,k\neq
k_{0}}^{K}A_{k}^{\dag}=\paren{\sum_{k=1}^{K}A_{k}}^{-1}-\sum_{k=1,k\neq
k_{0}}^{K}A_{k}^{\dag}.\] From here,
\begin{align*}
\paren{\sum_{k=1}^{K}A_{k}}A_{k_{0}}^{\dag}&=I_{\mathcal{H}_{2}}-\paren{\sum_{k=1}^{K}A_{k}}\paren{\sum_{k=1,k\neq
k_{0}}^{K}A_{k}^{\dag}}\\&=I_{\mathcal{H}_{2}}-\sum_{k=1,k\neq
k_{0}}^{K}A_{k}A_{k}^{\dag}=I_{\mathcal{H}_{2}}-\sum_{k=1,k\neq
k_{0}}^{K}P_{{\rm R}\!\paren{A_{k}}}\\&=I_{\mathcal{H}_{2}}-P_{{\rm
R}\!\paren{\sum_{k=1,k\neq k_{0}}^{K}A_{k}}}= P_{{\rm
N}\!\paren{\sum_{k=1,k\neq k_{0}}^{K}A_{k}^{*}}}
\end{align*}

\noindent and

\begin{align*}
A_{k_{0}}^{\dag}\paren{\sum_{k=1}^{K}A_{k}}&=I_{\mathcal{H}_{1}}-\paren{\sum_{k=1,k\neq
k_{0}}^{K}A_{k}^{\dag}}\paren{\sum_{k=1}^{K}A_{k}}\\&=I_{\mathcal{H}_{1}}-\sum_{k=1,k\neq
k_{0}}^{K}A_{k}^{\dag}A_{k}=I_{\mathcal{H}_{2}}-\sum_{k=1,k\neq
k_{0}}^{K}P_{{\rm R}\!\paren{A_{k}^{*}}}\\&=
I_{\mathcal{H}_{2}}-P_{{\rm R}\!\paren{\sum_{k=1,k\neq
k_{0}}^{K}A_{k}^{*}}}=P_{{\rm N}\!\paren{\sum_{k=1,k\neq
k_{0}}^{K}A_{k}}}.
\end{align*}
\end{proof}

It is important to note that in Definition~\ref{D MPI} and in
Theorem~\ref{T MPI caracterizacion}, $A^{\dag}$ is characterized as
the solution of systems of equations, or of operators equations
where the solution must satisfy some restrictions on its null space
or on its range space, whereas Corollary~\ref{C MPI expression sum
injective}, Corollary~\ref{C MPI expression sum surjective} and
Theorem~\ref{T MPI expression sum invertible} give characterizations
of $A^{\dag}$ as the solution of single equations.

The next result says that under hypotheses similar to the used for
the previous characterizations, any series of operators, which
converges in the operator norm, is a finite sum.

\begin{proposition}
Let $\set{A_{k}}_{k=1}^{\infty} \subset
\mathcal{BC}\!\paren{\mathcal{H}_{1}, \mathcal{H}_{2}}$. Assume that
${\rm R}\!\paren{A_{k}} \subseteq {\rm N}\!\paren{A_{k'}^{*}}$ and
${\rm R}\!\paren{A_{k}^{*}} \subseteq {\rm N}\!\paren{A_{k'}}$ for
each $k, k' = 1, 2, \ldots$, $k \neq k'$. Assume that the series
$\sum_{k=1}^{\infty}A_{k}$ converges in the operator norm. If ${\rm
N}\!\paren{\sum_{k=1}^{\infty}A_{k}}=\set{0}$ and ${\rm
R}\!\paren{\sum_{k=1}^{\infty}A_{k}}$ is closed, or ${\rm
R}\!\paren{\sum_{k=1}^{\infty}A_{k}}=\mathcal{H}_{2}$, then there
exists $K_{0} \in \mathbb{N}$ such that $A_{k}=0$ for each $k >
K_{0}$.
\end{proposition}
\begin{proof}
Assume that ${\rm N}\!\paren{\sum_{k=1}^{\infty}A_{k}}=\set{0}$ and
${\rm R}\!\paren{\sum_{k=1}^{\infty}A_{k}}$ is closed. Then
$\paren{\sum_{k=1}^{\infty}A_{k}}^{\dag}$ exists.

Let $K \in \mathbb{N}$. We have
\begin{equation}\label{E MPI series N 1}
\sum_{k=1}^{K}A_{k}=\sum_{k=1}^{\infty}A_{k}-\sum_{k=K+1}^{\infty}A_{k}.
\end{equation}

Let $g \in {\rm R}\paren{\sum_{k=K+1}^{\infty}A_{k}}$. There exists
$f \in \mathcal{H}_{1}$ such that
\[g=\sum_{k=K+1}^{\infty}A_{k}f=\sum_{k=K+1}^{\infty}A_{k}P_{\overline{{\rm
R}\paren{\sum_{k=K+1}^{\infty}A_{k}^{*}}}}f.\] Since $\overline{{\rm
R}\paren{\sum_{k=K+1}^{\infty}A_{k}^{*}}} \subseteq {\rm
N}\paren{\sum_{k=1}^{K}A_{k}}$,
\[g=\sum_{k=1}^{\infty}A_{k}P_{\overline{{\rm
R}\paren{\sum_{k=K+1}^{\infty}A_{k}^{*}}}}f.\] This shows that
\begin{equation}\label{E MPI series N 2}
{\rm R}\paren{\sum_{k=K+1}^{\infty}A_{k}} \subseteq {\rm
R}\paren{\sum_{k=1}^{\infty}A_{k}},
\end{equation}
for each $K \in \mathbb{N}$.

Let $K_{0} \in \mathbb{N}$ be such that
\[\norm{\paren{\sum_{k=1}^{\infty}A_{k}}^{\dag}\sum_{k=K+1}^{\infty}A_{k}}
< 1\] for each $K \geq K_{0}$. By Theorem~\ref{T MPI expression
sum}(ii), (\ref{E MPI series N 1}), (\ref{E MPI series N 2}) and
\cite[Lemma 3.3]{Ding-Huang (1994)}, if $K \geq K_{0}$, then
\[\bigoplus_{k=1}^{K}{\rm R}\!\paren{A_{k}}={\rm R}\!\paren{\sum_{k=1}^{K}A_{k}}={\rm
  R}\!\paren{\sum_{k=1}^{\infty}A_{k}}.\]

This implies that ${\rm R}\!\paren{A_{k}}=\set{0}$ for each $k >
K_{0}$.

Similarly, if ${\rm
R}\!\paren{\sum_{k=1}^{\infty}A_{k}}=\mathcal{H}_{2}$, we can apply
the previous reasoning to $\set{A_{k}^{*}}_{k=1}^{\infty}$ to
conclude that there exists $K_{0} \in \mathbb{N}$ such that
$A_{k}=0$ for each $k
> K_{0}$.
\end{proof}

\section{Formulations for matrices}\label{S Formulations for matrices}

In this section, we specialize previous results for finite matrices
and relate them with others in the literature. The elements of
$\mathbb{F}^{n}$ will be consider as column vectors, and if $x \in
\mathbb{F}^{n}$ then $x(i)$ denotes the $i$th component of $x$. The
elements of the standard basis of $\mathbb{F}^{n}$ will be denoted
by $e_{1}$, ..., $e_{n}$. We denote the vector with all its
components equal to $1$ with $e$. The set of $m \times n$ matrices
over $\mathbb{F}$ is denoted by $\mathcal{M}_{m,n}$. If $m=n$ we
write $\mathcal{M}_{n}$. If $A \in \mathcal{M}_{m,n}$, we denote the
entry $i,j$, the $i$th row and the $j$th column of $A$ with
$A(i,j)$, $A(i,:)$ and $A(:,j)$, respectively. We note that if $x, y
\in\mathbb{F}^{n}$, $x \neq 0$ and $y \neq 0$, then
$\paren{xy^{*}}^{\dag}=\frac{1}{\norm{x}^{2}\norm{y}^{2}}yx^{*}$.

We begin noting that considering $\set{A_{k}}_{k=1}^{K} \subset
\mathcal{M}_{m,n}$, Theorem~\ref{T MPI expression sum}(iv) is
\cite[Lemma 1.7]{Penrose (1955)}. For the case case $K=2$, there are
various papers that generalize \cite[Lemma 1.7]{Penrose (1955)}
considering weaker hypotheses. For example, if $A_{1}A_{2}^{*}=0$
and $C=\paren{I-A_{1}A_{1}^{\dag}}A_{2}$, \cite[Theorem 2]{Cline
(1965)} expresses $\paren{A_{1}+A_{2}}^{\dag}$ in terms of $A_{1}$,
$A_{2}$, $A_{1}^{*}$, $A_{2}^{*}$, $C$ and their Moore--Penrose
inverses. We can also mention \cite[Theorem 1]{Hung-Markham (1977)}
which generalize \cite[Theorem 2]{Cline (1965)} and expresses
$\paren{A_{1}+A_{2}}^{\dag}$ in terms of $A_{1}$, $A_{2}$,
$A_{1}^{*}$, $A_{2}^{*}$, other matrices and their Moore--Penrose
inverses. In particular, by \cite[Theorem 3]{Fill-Fishkind (1999)},
if $A_{1}, A_{2} \in \mathcal{M}_{n}$ and ${\rm
rank}\!\paren{A_{1}+A_{2}}={\rm rank}\!\paren{A_{1}}+{\rm
rank}\!\paren{A_{2}}$, then
\begin{equation}\label{E Fill-Fishkind}
\paren{A_{1}+A_{2}}^{\dag}=\paren{I-\paren{P_{{\rm R}\!\paren{A_{2}^{*}}}P_{{\rm
N}\!\paren{A_{1}}}}^{\dag}}A_{1}^{\dag}\paren{I-\paren{P_{{\rm
N}\!\paren{A_{1}^{*}}}P_{{\rm
R}\!\paren{A_{2}}}}^{\dag}}+\paren{P_{{\rm
R}\!\paren{A_{2}^{*}}}P_{{\rm
N}\!\paren{A_{1}}}}^{\dag}A_{2}^{\dag}\paren{P_{{\rm
N}\!\paren{A_{1}^{*}}}P_{{\rm R}\!\paren{A_{2}}}}^{\dag}.
\end{equation}
Equality (\ref{E Fill-Fishkind}) was proved for operators $A_{1},
A_{2} \in \mathcal{BC}\!\paren{\mathcal{H}_{1}, \mathcal{H}_{2}}$
such that ${\rm R}\!\paren{A_{1}} \cap {\rm R}\!\paren{A_{2}}={\rm
R}\!\paren{A_{1}^{*}} \cap {\rm R}\!\paren{A_{2}^{*}}=\set{0}$,
${\rm R}\!\paren{A_{1}+A_{2}}={\rm R}\!\paren{A_{1}}+{\rm
R}\!\paren{A_{2}}$ and ${\rm R}\!\paren{A_{1}^{*}+A_{2}^{*}}={\rm
R}\!\paren{A_{1}^{*}}+{\rm R}\!\paren{A_{2}^{*}}$ (see \cite[Theorem
5.2]{Arias-Corach-Maestripieri (1999)}). Note that if $A_{1}, A_{2}
\in \mathcal{BC}\!\paren{\mathcal{H}_{1}, \mathcal{H}_{2}}$ satisfy
the hypotheses of Theorem~\ref{T MPI expression sum}, they also
satisfy the hypotheses of \cite[Theorem
5.2]{Arias-Corach-Maestripieri (1999)} and from (\ref{E
Fill-Fishkind}) we get
$\paren{A_{1}+A_{2}}^{\dag}=A_{1}^{\dag}+A_{2}^{\dag}$. In
\cite{Sivakumar (2020), Baksalary-Sivakumar-Trenkler (2022)} are
considered sufficient conditions independent of the conditions of
\cite[Lemma 1.7]{Penrose (1955)} to have
$\paren{A_{1}+A_{2}}^{\dag}=A_{1}^{\dag}+A_{2}^{\dag}$ (see also
Remark~\ref{R Theorem 3.2 Sivakumar (2020) conditions not
necessary}). For $K$ arbitrary, we have \cite{Tian (1998), Tian
(2001), Tian (2005)} where $\paren{\sum_{k=1}^{K}A_{k}}^{\dag}$ is
expressed in terms of the Moore--Penrose inverse of block circulant
matrices.

The following particular case of Theorem~\ref{T MPI expression sum}
permits us to compute $A^{\dag}$ using the equality
$A^{\dag}=\paren{A+B}^{\dag}-B^{\dag}$ having $B^{\dag}$ an explicit
expression.
\begin{theorem}\label{T MPI expression}
Let $A \in \mathcal{M}_{m,n}$, $r = {\rm rank}\!\paren{A}$, $q =
{\rm min}\{m,n\}$ and $r \leq q' \leq q$. Let $\set{f_{1}, \ldots,
f_{q'-r}}$ be an orthonomal subset of ${\rm N}\!\paren{A}$,
$\set{g_{1}, \ldots, g_{q'-r}}$ be an orthonomal subset of ${\rm
N}\!\paren{A^{*}}$. Let $\set{d_{k}}_{k=1}^{q'-r} \subset
\mathbb{F}\setminus \set{0}$. Then
\begin{equation}
A^{\dag}=\paren{A+\sum_{k=1}^{q'-r}d_{k}g_{k}f_{k}^{*}}^{\dag}-\sum_{k=1}^{q'-r}\frac{1}{d_{k}}f_{k}g_{k}^{*}.
\end{equation}
\end{theorem}

As a consequence of Theorem~\ref{T MPI expression}, we get the
following result which can be viewed as a particular case of
Theorems~\ref{T MPI expression sum injective}(i) and~\ref{T MPI
expression sum surjective}(i):
\begin{corollary}\label{C MPI expression full rank}
Let $A \in \mathcal{M}_{m,n}$, $r = {\rm rank}\!\paren{A}$ and $q =
{\rm min}\{m,n\}$. Let $\set{f_{1}, \ldots, f_{q-r}}$ be an
orthonomal subset of ${\rm N}\!\paren{A}$ and $\set{g_{1}, \ldots,
g_{q-r}}$ be an orthonomal subset of ${\rm N}\!\paren{A^{*}}$. Let
$\set{d_{k}}_{k=1}^{q-r} \subset \mathbb{F}\setminus \set{0}$. Then:
\begin{enumerate}
 \item[(i)] If $m \geq n$, then $A^{\dag}=(A^{*}A+\sum_{k=1}^{q-r}d_{k}^{2}f_{k}f_{k}^{*})^{-1}(A^{*}+\sum_{k=1}^{q-r}d_{k}f_{k}g_{k}^{*})-\sum_{k=1}^{q-r}\frac{1}{d_{k}}f_{k}g_{k}^{*}$.
 \item[(ii)] If $n \geq m$, then $A^{\dag}=(A^{*}+\sum_{k=1}^{q-r}d_{k}f_{k}g_{k}^{*})(AA^{*}+\sum_{k=1}^{q-r}d_{k}^{2}g_{k}g_{k}^{*})^{-1}-\sum_{k=1}^{q-r}\frac{1}{d_{k}}f_{k}g_{k}^{*}$.
 \item[(iii)] If $m=n$, then $A^{\dag}=(A+\sum_{k=1}^{q-r}d_{k}g_{k}f_{k}^{*})^{-1}-\sum_{k=1}^{q-r}\frac{1}{d_{k}}f_{k}g_{k}^{*}$.
\end{enumerate}
\end{corollary}

\begin{remark}\label{R results literature}
{\rm From Corollary~\ref{C MPI expression full rank}(iii), we obtain
results that appear in the literature for the case $m=n$. First, we
note that the result of Corollary~\ref{C MPI expression full
rank}(iii) appears in \cite[Exercise 53]{Ben-Israel-Greville
(2003)}. The proof given there consists in a verification of the
equality
$\paren{A+\sum_{k=1}^{q-r}d_{k}g_{k}f_{k}^{*}}\paren{A^{\dag}+\sum_{k=1}^{q-r}\frac{1}{d_{k}}f_{k}g_{k}^{*}}=I$.

If $A \in \mathcal{M}_{n}$ is a normal matrix and $d_{k}=1$ for each
$k=1, \ldots, q-r$, from Corollary~\ref{C MPI expression full
rank}(iii) we obtain \cite[Theorem 1]{Bell (1981)} which is used to
give expressions for the Moore--Penrose inverses of circulant
matrices. The proof given in \cite{Bell (1981)} is based on the
spectral decomposition of $A$.

If $A \in \mathcal{M}_{n}$ is symmetric, $r=n-1$, $d_{1}=\alpha n$,
$\alpha \neq 0$ and $f_{1}=g_{1}=\frac{1}{\sqrt{n}}e \in {\rm
N}\!\paren{A}$, from Corollary~\ref{C MPI expression full rank}(iii)
we obtain \cite[Theorem 2.1]{Plonka-Hoffmanny-Weickertz (2016)}. In
\cite{Plonka-Hoffmanny-Weickertz (2016)}, the authors first prove
that $A+ \alpha ee^{t}$ is nonsingular by showing that all its
eigenvalues are nonzero. Then, they prove that $X=\paren{A+ \alpha
ee^{t}}^{-1}-\frac{1}{\alpha n^{2}}ee^{t}$ verifies the equations in
Definition~\ref{D MPI} and consequently $X=A^{\dag}$.}
\end{remark}

In the finite dimensional case we can get a result similar to
Theorem~\ref{T MPI expression sum injective}, Theorem~\ref{T MPI
expression sum surjective} and Theorem~\ref{T MPI expression sum
invertible} but with weaker assumptions.
\begin{theorem}\label{T MPI expression full rank comp1}
Let $A, B \in \mathcal{M}_{m,n}$. Then:
\begin{enumerate}
 \item[(i)] If ${\rm R}\!\paren{B^{*}}={\rm N}\!\paren{A}$, then $A^{*}A+B^{*}B$ is invertible and $A^{\dag}$ is the
unique solution of the equation $\paren{A^{*}A+B^{*}B}X=A^{*}$.

 \item[(ii)] If ${\rm R}\!\paren{B} = {\rm N}\!\paren{A^{*}}$, then $AA^{*}+BB^{*}$ is invertible and $A^{\dag}$ is the
unique solution of the equation $X\paren{AA^{*}+BB^{*}}=A^{*}$.

 \item[(iii)] If $m=n$, ${\rm R}\!\paren{B^{*}} = {\rm N}\!\paren{A}$ and ${\rm R}\!\paren{B} \subseteq {\rm N}\!\paren{A^{*}}$ (or, ${\rm R}\!\paren{B^{*}} \subseteq {\rm N}\!\paren{A}$ and ${\rm R}\!\paren{B} = {\rm N}\!\paren{A^{*}}$),
then $A+B$ is invertible, $A^{\dag}=\paren{A+B}^{-1}-B^{\dag}$ and
$A^{\dag}$ is the unique solution of any of the equations
$(A+B)X=P_{{\rm N}\!\paren{B^{*}}}$ and $X(A+B)=P_{{\rm
N}\!\paren{B}}$.
\end{enumerate}
\end{theorem}
\begin{proof}
Parts (i) and (ii) follow from \cite[Equalities (8) and (9)]{Ji
(2005)}.

Assume that $m=n$. If ${\rm R}\!\paren{B^{*}} = {\rm N}\!\paren{A}$
and ${\rm R}\!\paren{B} \subseteq {\rm N}\!\paren{A^{*}}$, by
Theorem~\ref{T MPI expression sum injective}(iii), $A+B$ is
injective and hence invertible. If ${\rm R}\!\paren{B^{*}} \subseteq
{\rm N}\!\paren{A}$ and ${\rm R}\!\paren{B} = {\rm
N}\!\paren{A^{*}}$, by Theorem~\ref{T MPI expression sum
surjective}(iii), $A+B$ is surjective and hence invertible. By
Theorem~\ref{T MPI expression sum}(iv),
$A^{\dag}=\paren{A+B}^{-1}-B^{\dag}$. The rest of part (iii) follows
from Theorem~\ref{T MPI expression sum invertible}.
\end{proof}

\begin{remark}
{\rm In \cite{Ji (2005)} appears the following variant of parts (i)
and (ii) of the previous theorem. Let $A \in \mathcal{M}_{m,n}$ and
$r = {\rm rank}\!\paren{A}$. Let $V \in M_{n-r,n}$, ${\rm
rank}(V)=n-r$ and ${\rm R}(V^{*})={\rm N}\!\paren{A}$. Let $W \in
M_{m,m-r}$, ${\rm rank}(W)=m-r$ and ${\rm R}(W)={\rm
N}\!\paren{A^{*}}$. By \cite[Theorem 3]{Ji (2005)},
$A^{\dag}=\paren{A^{*}A+V^{*}V}^{-1}A^{*}=A^{*}\paren{AA^{*}+WW^{*}}^{-1}$.

Theorem 3 in \cite{Ji (2005)} is used to obtain condensed Cramer
rules for the minimal-norm least-squares solution $x=A^{\dag}b$ of
linear equations $Ax=b$ and to give condensed determinantal
expressions for $A^{\dag}$, $AA^{\dag}$ and $A^{\dag}A$. This
theorem in \cite{Ji (2005)} is proved given before an explicit
expression for the $\{2\}$-inverse of $A$ with range $T$ and null
space $S$ (see \cite[Theorem 2]{Ji (2005)}). Then the expressions
for $A^{\dag}$ are obtained considering $T={\rm R}\!\paren{A^{*}}$
and $S={\rm N}\!\paren{A^{*}}$.}
\end{remark}

We finish this section showing that the sufficient conditions of
\cite[Theorem 3.2]{Sivakumar (2020)} (see also \cite[Proposition
2.3]{Baksalary-Sivakumar-Trenkler (2022)}) and of Theorem~\ref{T MPI
expression sum} are not necessary. This will also show the relation
of  Theorem~\ref{T MPI expression sum} with the singular value
decomposition. We recall first a property. Let $A = V \Sigma W^{*}$
where $V \in \mathcal{M}_{m}$ and $W \in \mathcal{M}_{n}$ are
unitary matrices, and
\begin{equation}\label{E Sigma}
\Sigma=\left(
         \begin{array}{cc}
           \Sigma_{r} & O_{r,n-r} \\
           O_{m-r,r} & O_{m-r,n-r} \\
         \end{array}
       \right) \in \mathcal{M}_{m,n}.
\end{equation}
with $O_{k,l}$ a zero matrix in $\mathcal{M}_{k,l}$ and $\Sigma_{r}$
a diagonal matrix in $\mathcal{M}_{r}$ with nonzero diagonal
elements. Let $\Sigma^{\dag}$ obtained from $\Sigma$ by first
replacing each nonzero element with its inverse and then
transposing. Then $A^{\dag}=W \Sigma^{\dag} V^{*}$.

\begin{remark}\label{R Theorem 3.2 Sivakumar (2020) conditions not necessary}
{\rm Let $A, B \in \mathcal{M}_{n}$. By \cite[Theorem 3.2]{Sivakumar
(2020)}, if $AB^{*} + BB^{*} = 0$ and $B^{*}A + B^{*}B = 0$ (see
\cite{Baksalary-Sivakumar-Trenkler (2022)} for details about these
conditions), then $\paren{A+B}^{\dag}=A^{\dag}+B^{\dag}$. Now, if
${\rm R}\!\paren{B^{*}} \subseteq {\rm N}\!\paren{A}$, ${\rm
R}\!\paren{A^{*}} \subseteq {\rm N}\!\paren{B}$, ${\rm R}(B)
\subseteq {\rm N}\!\paren{A^{*}}$, ${\rm R}\!\paren{A} \subseteq
{\rm N}\!\paren{B^{*}}$, $BB^{*} \neq0$ and $B^{*}B \neq 0$, then
$AB^{*} + BB^{*} \neq 0$ and $B^{*}A + B^{*}B \neq 0$ and, by
Theorem~\ref{T MPI expression sum},
$\paren{A+B}^{\dag}=A^{\dag}+B^{\dag}$. We are going to see that $A,
B \in \mathcal{M}_{n}$ with the previous properties exist.

Let $A \in \mathcal{M}_{n}$ be arbitrary, $r = {\rm
rank}\!\paren{A}$, $q = {\rm min}\{n\}$ and $r \leq q' \leq q$. Let
$A = V \Sigma W^{*}$ be a singular value decomposition of $A$ where
$\Sigma$ is as in (\ref{E Sigma}) and the diagonal elements
$\sigma_{1}, \ldots, \sigma_{r}$ of $\Sigma_{r}$ are the
\emph{singular values} of $A$ (see e.g. \cite[7.3.P7]{Horn-Johnson
(2013)}). Let $B=VEW^{*}$ where
\[E=\left(
      \begin{array}{ccc}
        O_{r,r} & O_{r,q'-r} & O_{r,n-q'} \\
        O_{q'-r,r} & E_{q'-r} & O_{q'-r,n-q'} \\
        O_{n-q',r} & O_{n-q',q'-r} & O_{n-q',n-q'} \\
      \end{array}
    \right)\] and $E_{q'-r}$ is a diagonal matrix in $M_{q'-r}$ with positive diagonal elements $e_{1}$, ..., $e_{q'-r}$. Then ${\rm
R}\!\paren{B^{*}} \subseteq {\rm N}\!\paren{A}$, ${\rm
R}\!\paren{A^{*}} \subseteq {\rm N}\!\paren{B}$, ${\rm R}\!\paren{B}
\subseteq {\rm N}\!\paren{A^{*}}$, ${\rm R}\!\paren{A} \subseteq
{\rm N}\!\paren{B^{*}}$, $BB^{*} \neq0$ and $B^{*}B \neq 0$ and
\[\paren{A+B}^{\dag}=W\paren{\Sigma+
E}^{\dag}V^{*}=W\paren{\Sigma^{\dag}+E^{\dag}}V^{*}=W\Sigma^{\dag}V^{*}+WE^{\dag}V^{*}=A^{\dag}+B^{\dag}.\]
This shows that the sufficient conditions of \cite[Theorem
3.2]{Sivakumar (2020)} are not necessary.}
\end{remark}
\begin{remark}\label{R Theorem T MPI expression sum conditions not necessary}
{\rm Let $\alpha, \beta, \gamma \in \mathbb{C}\setminus\set{0}$ be
such that
$\paren{\alpha+\beta+\gamma}^{-1}=\alpha^{-1}+\beta^{-1}+\gamma^{-1}$
(e.g., $\alpha=\beta=-\gamma$). Let $V \in \mathcal{M}_{m}$ and $W
\in \mathcal{M}_{n}$ be unitary matrices. Let $\Sigma_{\alpha}$,
$\Sigma_{\beta}$ and $\Sigma_{\gamma}$ as in (\ref{E Sigma}) with
the $r$ nonzero entries equal to $\alpha$, $\beta$ and $\gamma$,
respectively. Thus, $\paren{\Sigma_{\alpha}+ \Sigma_{\beta}+
\Sigma_{\gamma}}^{\dag}=\Sigma_{\alpha}^{\dag}+
\Sigma_{\beta}^{\dag}+ \Sigma_{\gamma}^{\dag}$ and if $A = V
\Sigma_{\alpha} W^{*}$, $B = V \Sigma_{\beta} W^{*}$ and $C = V
\Sigma_{\gamma} W^{*}$, then
$\paren{A+B+C}^{\dag}=A^{\dag}+B^{\dag}+C^{\dag}$. In this case,
${\rm R}\!\paren{A}={\rm R}\!\paren{B}={\rm R}\!\paren{C}$, ${\rm
N}\!\paren{A}={\rm N}\!\paren{B}={\rm N}\!\paren{C}$, ${\rm
R}\!\paren{A^{*}}={\rm R}\!\paren{B^{*}}={\rm R}\!\paren{C^{*}}$ and
${\rm N}\!\paren{A^{*}}={\rm N}\!\paren{B^{*}}={\rm
N}\!\paren{C^{*}}$. This example shows that the conditions of
Theorem~\ref{T MPI expression sum} are not necessary.

Note that $(\alpha+\beta)^{-1}=\alpha^{-1}+\beta^{-1}$ if and only
if $\alpha^{2}+\alpha\beta+\beta^{2}=0$, or equivalently,
$\alpha=\paren{\frac{-1\pm i \sqrt{3}}{2}}\beta$. Hence $\alpha$ and
$\beta$ cannot be both real numbers. See \cite{Boman-Uhlig (2002),
DAngelo (2012)} for interesting details about the equality
$\paren{\alpha+\beta}^{-1}=\alpha^{-1}+\beta^{-1}$.}
\end{remark}

\section{Moore--Penrose inverse of circulant matrices}\label{S MPI of
circulant matrices}

In this section, we consider circulant matrices of order $n \geq 2$,
$C={\rm circ}\!\paren{c}$ where $c \in \mathbb{C}^{n}$ (see
\cite{Davis (1994), Aldrovandi (2001), Gray (2001)}). For example,
if $n=2, 3$ we have, \[C=\left(
                       \begin{array}{cc}
                         c(1) & c(2) \\
                         c(2) & c(1) \\
                       \end{array}
                     \right),\,\,
\,\,\,C=\left(
          \begin{array}{ccc}
            c(1) & c(2) & c(3) \\
            c(3) & c(1) & c(2) \\
            c(2) & c(3) & c(1) \\
          \end{array}
        \right).
\] If $\Pi={\rm circ}\!\paren{0,1,0,\ldots,0}$ then
$\Pi^{-1}=\Pi^{n-1}=\Pi^{t}={\rm circ}\!\paren{0,\ldots,0,1}$ and
$C={\rm circ}\!\paren{c}=\sum_{k=0}^{n-1}c(k+1)\Pi^{k}$. If $\rho:
\mathbb{C}^{n} \rightarrow \mathbb{C}^{n}$ is given by
$\paren{\rho(c)}\!(1)=c(1)$ and $\paren{\rho(c)}(k)=c(n-k+2)$ for
$k=2, \ldots, n$, then $C^{t}={\rm circ}\!\paren{\rho(c)}$.

If ${\rm circ}\!\paren{c}={\rm circ}\!\paren{a}{\rm
circ}\!\paren{b}$, then ${\rm circ}\!\paren{c}={\rm
circ}\!\paren{b}{\rm circ}\!\paren{a}$ and
\begin{equation}\label{E producto de circulantes}
c(l)=\sum_{k=1}^{l}a(k)b(l-k+1)+\sum_{k=l+1}^{n}a(k)b(n+l-k+1)
\end{equation} for $l = 1, \ldots, n$, or equivalently, $c = {\rm circ}\!\paren{\rho(a)}b$.

\begin{remark}\label{R MPI rol}
{\rm If $0 \leq l \le n-1$ and $C$ is a circulant matrix, then ${\rm
N}\!\paren{\Pi^{l}C}={\rm N}\!\paren{C}$ and, by the reverse-order
law for the Moore--Penrose inverse (see, e.g., \cite[Chapter 4, Ex.
22]{Ben-Israel-Greville (2003)}),
$\paren{\Pi^{l}C}^{\dag}=C^{\dag}\Pi^{n-l}$.}
\end{remark}

In \cite{Carmona-Encinas-Jimenez-Mitjana (2021),
Carmona-Encinas-Jimenez-Mitjana (2022)}, the coefficients of the
inverse and the group inverse of a circulant matrix depending on up
to four complex parameters, i.e. ${\rm
circ}\!\paren{a,b,c,d,\ldots,d}$, are expressed in terms of
functions $k_{j}(a,b,c,d)$, $j = 1, \ldots,n$. In particular, in the
case of four parameters these functions involves Chebyshev
polynomials. The group inverse of a circulant matrix coincides with
its Moore--Penrose inverse. The techniques used in these papers are
related with the solution of boundary value problems associated to
second-order linear difference equations. Here we use results of the
previous sections to obtain properties and explicit expressions of
the Moore--Penrose inverse of circulant matrices.

Consider circulant matrices of the form
$C=\alpha\Pi^{k-1}+\beta\Pi^{k}$ where $\alpha, \beta\in
\mathbb{C}\setminus\set{0}$ and $1 \leq k < n$. Then
$C=\Pi^{k-1}\paren{\alpha I+\beta\Pi}$. Similarly, $C=\alpha
I+\beta\Pi^{n-1}=\Pi^{n-1}\paren{\beta I+\alpha\Pi}$. Hence, by
Remark~\ref{R MPI rol}, in both cases in order to study $C^{\dag}$
it is sufficient to consider circulant matrices of the form
$C=\alpha I+\beta\Pi$. Part (ii) of the following lema can be
obtained from \cite[Theorem
2.3]{Carmona-Encinas-Gago-Jimenez-Mitjana (2015)} with parameters
$a=2$, $b=0$ and $c=1$. We include here a direct brief proof.
\begin{lemma}\label{L circ 2,0,1, ...,1}
The following assertions hold:
\begin{enumerate}
  \item[(i)] If $n$ even, $C_{1,1}={\rm
circ}\!\paren{1,1, 0, \ldots, 0}$, $v=\paren{1,-1,1,-1,\ldots,1,-1}$
and $w \in \mathbb{R}^{n}$ is given by
\[w(k)=\paren{-1}^{k+1}\paren{n^{2}-\paren{2k-1}n+2}\] for $k=1,
\ldots, n$, then $C_{1,1}+v*v^{t}={\rm
circ}\!\paren{2,0,1,-1,\ldots,1,-1}$ is invertible and
\[\paren{C_{1,1}+v*v^{t}}^{-1}=\frac{1}{2n^{2}}{\rm
circ}\!\paren{w}.\]

  \item[(ii)] If $C_{1,-1}={\rm circ}\!\paren{1,-1,
0, \ldots, 0}$, then $C_{1,-1}+ee^{t}={\rm
circ}\!\paren{2,0,1,\ldots,1}$ is invertible and
\[\paren{C_{1,-1}+ee^{t}}^{-1}=-\frac{1}{2n}{\rm
circ}\!\paren{1,3,5,\ldots,2n-1}+\frac{n^{2}+2}{2n^{2}}ee^{t}.\]
\end{enumerate}
\end{lemma}
\begin{proof}
(i) Let ${\rm circ}\!\paren{a}:=\paren{C_{1,1}+v*v^{t}}{\rm
circ}\!\paren{w}$. Using (\ref{E producto de circulantes}) we
obtain,
\begin{align*}
a(1)&=2\paren{n^{2}-n+2}+\sum_{k=2}^{n-1}\paren{-1}^{n-k+3}\paren{-1}^{k+1}\paren{n^{2}-\paren{2k-1}n+2}\\
&=2\paren{n^{2}-n+2}+\paren{n-2}\paren{n^{2}+n+2}-2n\paren{\frac{(n-1)n}{2}-1}=2n^{2},
\end{align*}
and for $l = 2, \ldots, n$,
\begin{align*}
a(l)=&\sum_{k=1}^{l-2}\paren{-1}^{l-k+2}\paren{-1}^{k+1}\paren{n^{2}-\paren{2k-1}n+2}+2\paren{-1}^{l+1}\paren{n^{2}-\paren{2l-1}n+2}+\\
&+\sum_{k=l+1}^{n}\paren{-1}^{n+l-k+2}\paren{-1}^{k+1}\paren{n^{2}-\paren{2k-1}n+2}\\
=&\paren{-1}^{l+1}\paren{\sum_{k=1}^{n}\paren{n^{2}-\paren{2k-1}n+2}+\paren{n^{2}-\paren{2l-1}n+2}-\paren{n^{2}-\paren{2\paren{l-1}-1}n+2}}\\
=&\paren{-1}^{l+1}\paren{n\paren{n^{2}+n+2}-2n\frac{n(n+1)}{2}-2n}=0.
\end{align*}
This shows that $\paren{C_{1,1}+v*v^{t}}^{-1}=\frac{1}{2n^{2}}{\rm
circ}\!\paren{w}$.

(ii) Let $B=-\frac{1}{2n}{\rm
circ}\!\paren{1,3,5,\ldots,2n-1}+\frac{n^{2}+1}{2n^{2}}ee^{t}$.
Using that $C_{1,-1}e=0$ and (\ref{E producto de circulantes}), we
obtain
\begin{align*}
\paren{C_{1,-1}+ee^{t}}B&=-\frac{1}{2n}C_{1,-1}{\rm
circ}\!\paren{1,3,5,\ldots,2n-1}-\frac{1}{2n}\sum_{k=1}^{n}\paren{2k-1}ee^{t}+\frac{n^{2}+2}{2n}ee^{t}\\
&=-\frac{1}{2n}{\rm
circ}\!\paren{2-2n,2,2,\ldots,2}+\frac{1}{n}ee^{t}={\rm
circ}\!\paren{1,0,\ldots,0}=I.
\end{align*}
Hence $B=C^{-1}$.
\end{proof}
The following proposition can be derived from \cite[Theorem
3.4.]{Carmona-Encinas-Jimenez-Mitjana (2021)} considering parameters
$a=\alpha$, $b=\beta$, and $c=0$. Here, we present a short proof
based on the previous lemma and Corollary~\ref{C MPI expression full
rank}(iii), giving the explicit expressions for the Moore--Penrose
inverses.
\begin{proposition}\label{P MPI circulant dos}
Let $C=\alpha I+\beta\Pi$ with $\alpha, \beta\in
\mathbb{C}\setminus\set{0}$. Then $C$ is singular if and only if
$\beta^{n}=\paren{-1}^{n}\alpha^{n}$ and the following assertions
hold:
\begin{enumerate}
  \item[(i)] If $n$ is even and $\alpha=\beta$, then ${\rm N}\!\paren{C}={\rm
span}\set{\paren{1,-1,1,-1,\ldots,1,-1}}$ and
\[C^{\dag}=\frac{1}{\alpha n^{2}}\paren{\frac{1}{2}{\rm circ}\!\paren{w}-{\rm
circ}\!\paren{1,-1,1,-1,\ldots,1,-1}},\] where $w \in
\mathbb{R}^{n}$ is given by
$w(k)=\paren{-1}^{k+1}\paren{n^{2}-\paren{2k-1}n+2}$ for $k=1,
\ldots, n$.
  \item[(ii)] If $\alpha=-\beta$, then ${\rm N}\!\paren{C}={\rm
span}\set{\paren{1,1,\ldots,1}}$ and
\begin{align*}
C^{\dag}&=\frac{1}{2\alpha}\paren{-\frac{1}{n}{\rm
circ}\!\paren{1,3,5,\ldots,2n-1}+ee^{t}}\\
&=\frac{1}{2nc_{1}}{\rm circ}\!\paren{n-1,n-3,n-5,\ldots,2-n,1-n}.
\end{align*}
\end{enumerate}
\end{proposition}
\begin{proof}
We have $Cx=0$ if and only if
$x(k+1)=\paren{-1}^{k}\paren{\frac{\alpha}{\beta}}^{k}x(1)$ for each
$k=1, \ldots, n-1$ and $x(n)=-\frac{\beta}{\alpha}x(1)$. Hence, $C$
is singular if and only if $\beta^{n}=\paren{-1}^{n}\alpha^{n}$.

(i) If $n$ is even and $\alpha=\beta$, then ${\rm N}\!\paren{C}={\rm
span}\set{\paren{1,-1,1,-1,\ldots,1,-1}}$. Let $C_{1,1}$ and $v$ as
in Lemma~\ref{L circ 2,0,1, ...,1}. By Corollary~\ref{C MPI
expression full rank}(iii) and Lemma~\ref{L circ 2,0,1, ...,1},
\begin{align*}
C^{\dag}&=\frac{1}{\alpha}C_{1,1}^{\dag}=\frac{1}{\alpha}\paren{\paren{C_{1,1}+vv^{t}}^{-1}-\frac{1}{\norm{v}^{2}}vv^{t}}\\
&=\frac{1}{\alpha n^{2}}\paren{\frac{1}{2}{\rm circ}\!\paren{w}-{\rm
circ}\!\paren{1,-1,1,-1,\ldots,1,-1}}.
\end{align*}

(ii) If $\alpha=-\beta$, then ${\rm N}\!\paren{C}={\rm
span}\set{\paren{1,1,\ldots,1}}$. Consider $C_{1,-1}$ as in
Lemma~\ref{L circ 2,0,1, ...,1}. By Corollary~\ref{C MPI expression
full rank}(iii) and Lemma~\ref{L circ 2,0,1, ...,1},
\begin{align*}
C^{\dag}&=\frac{1}{\alpha}C_{1,-1}^{\dag}=\frac{1}{\alpha}\paren{\paren{C_{1,-1}+ee^{t}}^{-1}-\frac{1}{n^{2}}ee^{t}}\\
&=\frac{1}{\alpha}\paren{-\frac{1}{2n}{\rm
circ}\!\paren{1,3,5,\ldots,2n-1}+\frac{n^{2}+2}{2n^{2}}ee^{t}-\frac{1}{n^{2}}ee^{t}}\\
&=\frac{1}{2\alpha}\paren{-\frac{1}{n}{\rm
circ}\!\paren{1,3,5,\ldots,2n-1}+ee^{t}}.
\end{align*}
\end{proof}

The Fourier matrix of order $n$ denoted with $F$ is given by
$F(k,l)=\frac{1}{\sqrt{n}}e^{-\frac{2 \pi i}{n}(k-1)(l-1)}$. If
$C={\rm circ}\!\paren{c}$, then $C$ is diagonalizable by $F$,
\[C=\overline{F}\Lambda F,\,\,\,\,
C^{\dag}=\overline{F}\Lambda^{\dag} F,\]
\begin{equation}\label{E MPI circulant lambda c}
\lambda=\sqrt{n}\overline{F}c \text { and }
c=\frac{1}{\sqrt{n}}F\lambda.
\end{equation}
where $\lambda$ is the diagonal of $\Lambda$. As a consequence of
Theorem~\ref{T MPI expression sum} and (\ref{E MPI circulant lambda
c}), we obtain the next theorem that provides a way to obtain the
(Moore--Penrose) inverse of circulant matrices in terms of the
(Moore--Penrose) inverses of other circulant matrices. The
\emph{support} of a vector $c$, denoted by ${\rm supp}(c)$, are the
indices of the nonzero components of $c$.

\begin{theorem}\label{T MPI expression sum circulant}
Let $\set{c_{k}}_{k=1}^{K} \subseteq \mathbb{C}^{n}$,
$\lambda_{k}=\sqrt{n}\overline{F}c_{k}$ and
$\widetilde{\lambda_{k}}=\sqrt{n}\overline{F}\rho\!\paren{c_{k}}$
for each $k = 1, \ldots, K$. If ${\rm supp}\!\paren{\lambda_{k}}
\cap {\rm supp}\!\paren{\widetilde{\lambda_{k'}}} = \emptyset$ for
each $k, k' = 1, \ldots, K$, $k \neq k'$, then
\[\paren{{\rm
circ}\!\paren{\sum_{k=1}^{K}c_{k}}}^{\dag}=\sum_{k=1}^{K}\paren{{\rm
circ}\!\paren{c_{k}}}^{\dag}.\] Moreover, if $\bigcup_{k=1}^{K}{\rm
supp}\!\paren{\lambda_{k}}=\set{1, \ldots, n}$, then ${\rm
circ}\!\paren{\sum_{k=1}^{K}c_{k}}$ is invertible.
\end{theorem}

\begin{example}\label{Ex MPI circulant mas menos}
Let $n$ be even. Let $\alpha, \beta \in \mathbb{C}\setminus\set{0}$.
Let $c_{\beta}=\!\paren{0, \ldots, 0, \beta,\beta,0, \ldots, 0}$
where $\beta$ appears in the $l$-th position. Then
\[\lambda_{\beta}=\sqrt{n}\overline{F}c_{\beta}=\beta\paren{e^{\frac{2
\pi i}{n}(k-1)(l-1)}+e^{\frac{2 \pi
i}{n}(k-1)l}}_{k=1}^{n}=\beta\paren{e^{\frac{2 \pi
i}{n}(k-1)l}\paren{e^{-\frac{2 \pi i}{n}(k-1)}+1}}_{k=1}^{n}.\]
Hence $\lambda_{\beta}(\frac{n}{2}+1)=0$ and
$\lambda_{\beta}(k)\neq0$ for $k \neq \frac{n}{2}+1$. If
$\widetilde{\lambda_{\alpha}}=\alpha n e_{\frac{n}{2}+1}$, then
\[c_{\alpha}:=\frac{1}{\sqrt{n}}F\widetilde{\lambda_{\alpha}}=\alpha
\paren{1,-1,\ldots,1,-1}=\rho(c_{\alpha}).\] By
Theorem~\ref{T MPI expression sum circulant},
\[\paren{{\rm circ}\!\paren{c_{\alpha}+c_{\beta}}}^{-1}=\paren{{\rm circ}\!\paren{c_{\alpha}}}^{\dag}+\paren{{\rm circ}\!\paren{c_{\beta}}}^{\dag}.\] If
$c=\paren{1,-1,\ldots,1,-1}$, then $\paren{{\rm
circ}\!\paren{c_{\alpha}}}^{\dag}=\frac{1}{\alpha
\norm{c}^{4}}vv^{t}=\frac{1}{\alpha n^{2}}{\rm circ}\!\paren{c}$. We
also have \[\paren{{\rm
circ}\!\paren{c_{\beta}}}^{\dag}=\frac{1}{\beta}\paren{{\rm
circ}\!\paren{1,1,0, \ldots, 0}}^{\dag}\Pi^{n-l+1}\] and the
explicit expression of $\paren{{\rm circ}\!\paren{1,1,0, \ldots,
0}}^{\dag}$ appears in Proposition~\ref{P MPI circulant dos}(i).
\end{example}

We have ${\rm supp}\!\paren{\sqrt{n}\overline{F}e}=\set{1}$, and if
$a \in \mathbb{C}^{n}$ is such that $\sum_{j=1}^{n}a(j)=0$, then $1
\notin {\rm supp}\!\paren{\sqrt{n}\overline{F}a}$. Hence, by
Theorem~\ref{T MPI expression sum circulant}, we get the following
proposition.
\begin{proposition}\label{P MPI circulant suma cero}
Let $c \in \mathbb{C}^{n}$. The following assertions hold:
\begin{enumerate}
  \item[(i)] If $\sum_{j=1}^{n}c(j)\neq0$, then
\[{\rm circ}\!\paren{c}^{\dag}={\rm
circ}\!\paren{c(1)-\frac{\sum_{j=1}^{n}c(j)}{n}, \ldots,
c(n)-\frac{\sum_{j=1}^{n}c(j)}{n}}^{\dag}+\frac{1}{n\sum_{j=1}^{n}c(j)}ee^{t}.\]
  \item[(ii)] If $\sum_{j=1}^{n}c(j)=0$ and $\alpha\neq0$, then
\[{\rm circ}\!\paren{c}^{\dag}={\rm
circ}\!\paren{c(1)+\alpha, \ldots,
c(n)+\alpha}^{\dag}-\frac{1}{n^{2}\alpha}ee^{t}.\]
\end{enumerate}
\end{proposition}
By the previous proposition, it is very easy to obtain the
Moore--Penrose inverse in the case $\sum_{j=1}^{n}c(j)\neq0$ from
the case $\sum_{j=1}^{n}c(j)=0$, and vice versa. As a consequence,
in order to give explicit expressions, there is no need to
separately consider each of the two cases. This fact will simplify
future researchers in the subject.

Part (ii) of the following proposition can be used to compute
explicitly and easily the Moore--Penrose inverse of
\[{\rm
circ}\!\paren{a,\underbrace{b,\ldots,b}_{k},a,\underbrace{b,\ldots,b}_{k},
\ldots, a,\underbrace{b,\ldots,b}_{k}}\] for each $a, b \in
\mathbb{C}$ choosing $\alpha=\frac{a+kb}{k+1}$ and
$\beta=\frac{a-b}{k+1}$.

\begin{proposition}\label{P MPI circulant abbb}
Let $k$ and $q$ be positive integers and $n=q(k+1)$. The following
assertions hold:
\begin{enumerate}
  \item[(i)] If $C={\rm
circ}\!\paren{k,\underbrace{-1,\ldots,-1}_{k}, \ldots,
k,\underbrace{-1,\ldots,-1}_{k}}$, then $C^{2}=nC$ and
$C^{\dag}=\frac{1}{n^{2}}C$.
  \item[(ii)] If $\alpha, \beta \in \mathbb{C}\setminus\set{0}$, then \[{\rm
circ}\!\paren{\alpha+k\beta,\underbrace{\alpha-\beta,\ldots,\alpha-\beta}_{k},
\ldots,
\alpha+k\beta,\underbrace{\alpha-\beta,\ldots,\alpha-\beta}_{k}}^{\dag}=\frac{1}{\alpha
n^{2}}ee^{t}+\frac{1}{\beta n^{2}}C.\]
\end{enumerate}
\end{proposition}
\begin{proof}
For $j=1, k+2, 2k+3,3k+4, \ldots, n-k$,

\centerline{$C^{2}(1,j)=qk^{2}+qk=qk(k+1)=nk=nC(1,j)$.}

\noindent Otherwise,

\centerline{$C^{2}(1,j)=2q(-k)+n-2q=n-2q(k+1)=n-2n=-n=nC(1,j)$.}

\noindent  Hence, $C^{2}=nC$. The rest of part (i) follows from this
equality and Definition~\ref{D MPI}.

Taking into account that \[{\rm
circ}\!\paren{\alpha+k\beta,\underbrace{\alpha-\beta,\ldots,\alpha-\beta}_{k},
\ldots,
\alpha+k\beta,\underbrace{\alpha-\beta,\ldots,\alpha-\beta}_{k}}=\alpha
ee^{t}+\beta C,\] part (ii) follows from Proposition~\ref{P MPI
circulant suma cero}.
\end{proof}

\section{Moore--Penrose inverse of distance matrices of certain
graphs}\label{S MPI distance matrix graphs}

In this section, we give applications to distance matrices of
certain graphs. Specifically, we consider distance matrices of
weighted trees and of wheel graphs with an odd number of vertices.

\subsection{Moore--Penrose inverse of the distance matrix of a weighted tree}\label{SubS MPI distance matrix weighted tree}

Let $T = (V ,E)$ denotes a \emph{weighted tree} with the set of
\emph{vertices} $V=\set{1, \ldots , n}$ and the set $E$ of unordered
pairs \emph{edges} $(i, j)$, $i \neq j$. To each $i, j \in V$ is
assigned a \emph{weight} $w_{ij} \neq 0$ if $i \neq j$ and $(i, j)$
is an edge of $T$. If $i \neq j$ and $(i, j)$ is not an edge of $T$
then $w_{ij} = 0$. The \emph{Laplacian matrix} $L$ of $T$ is the $n
\times n$ positive definitive matrix given by $L(i,j)=-w_{ij}^{-1}$
if $i \neq j$ and $(i, j)$ is an edge of $T$, $L(i,j)=0$ if $i \neq
j$ and $(i, j)$ is not an edge of $T$, and $L(i,i)=-\sum_{j \neq
i}L(i,j)$. The \emph{distance matrix} $D$ of $T$ is the matrix with
$D(i,j)$ equal to the \emph{distance} between vertices $i$ and $j$,
defined to be the sum of the weights of the edges on the (unique)
$ij$-path. We set $D(i,i) = 0, i = 1, \ldots , n$. We denote the
degree of the vertex $i$ by $\delta(i)$ , $i = 1, \ldots , n$. Let
$\delta$ be the vector with components $\delta(1), \ldots ,
\delta(n)$. We set $\tau = 2e - \delta$.

Assume that $\sum_{j=1}^{n-1}w_{j}=0$ and that all the weights are
nonzero. By \cite[Theorem 11]{Kurata-Bapat (2016)},
\begin{equation}\label{E MPI D wtree L tau u}
D^{\dag}=-\frac{1}{2}L+u\tau^{t}+\tau u^{t},
\end{equation}
where
\begin{equation}\label{E MPI D wtree u}
u=\frac{1}{2}\paren{D^{\dag}e + \frac{e^{t}D^{\dag}e}{4}\tau}.
\end{equation}
From the proof of \cite[Theorem 11]{Kurata-Bapat (2016)}, ${\rm
N}(D)={\rm span}\{\tau\}$. Thus, by Corollary~\ref{C MPI expression
full rank}(iii),
\begin{equation}\label{E MPI D wtree 1}
D^{\dag}=\paren{D+\alpha\tau\tau^{t}}^{-1}-\frac{1}{\alpha\norm{\tau}^{4}}\tau\tau^{t}.
\end{equation}
or equivalently, multiplying both sides from the left by
$D+\alpha\tau\tau^{t}$, $D^{\dag}$ is the unique solution of the
equation
\begin{equation}\label{E MPI D wtree 2}
\paren{D+\alpha\tau\tau^{t}}X=I-\frac{1}{\norm{\tau}^{2}}\tau\tau^{t}.
\end{equation}
By Corollary~\ref{C 134 inverse}, the invertible matrix
$D+\alpha\tau\tau^{t}$ is a $\{1, 3, 4\}$-inverse of $D$.
Expressions (\ref{E MPI D wtree 1}) and (\ref{E MPI D wtree 2}) are
alternatives for (\ref{E MPI D wtree L tau u}) to obtain $D^{\dag}$.
In order to compute $u$ in (\ref{E MPI D wtree L tau u}), we note
that from (\ref{E MPI D wtree 1}),
\begin{equation}\label{E MPI D wtree Ddage}
D^{\dag}e=\paren{D+\alpha\tau\tau^{t}}^{-1}e-\frac{2}{\alpha\norm{\tau}^{4}}\tau,
\end{equation}
\noindent and from (\ref{E MPI D wtree 2}), $D^{\dag}e$ is the
unique solution of the equation

\centerline{$\paren{D+\alpha\tau\tau^{t}}x=e-\frac{2}{\norm{\tau}^{2}}\tau$,}

\noindent where $\alpha \neq 0$. If $\tau^{t}L\tau=\delta^{t}L\delta
\neq 0$, choosing $\alpha$ appropriately, from (\ref{E MPI D wtree
Ddage}) we get an expression of $u$ in (\ref{E MPI D wtree u}) that
do not involve $D^{\dag}$.

\begin{proposition}\label{P MPI D wtree u Ldelta}
Let $T$ be a weighted tree with distance matrix $D$, Laplacian
matrix $L$, degrees given by $\delta$ and $\tau = 2e - \delta$.
Assume that $\sum_{j=1}^{n-1}w_{j}=0$ and that all the weights are
nonzero. If $\tau^{t}L\tau \neq 0$, then the vector $u$ in (\ref{E
MPI D wtree u}) is given by
\begin{equation}\label{E MPI D wtree u e}
u=\frac{1}{2}\paren{\frac{1}{\norm{\tau}^{2}}L\tau-\frac{3\tau^{t}L\tau}{2\norm{\tau}^{4}}\tau}.
\end{equation}
\end{proposition}
\begin{proof}
By \cite[Lemma 9]{Kurata-Bapat (2016)}, $DL=e\tau^{t}-2I$. Hence,
\[\paren{D+\frac{2}{\tau^{t}L\tau}\tau\tau^{t}}\paren{\frac{1}{\norm{\tau}^{2}}L\tau}=e.\] From (\ref{E MPI D wtree Ddage}) with
$\alpha=\frac{2}{\tau^{t}L\tau}$ and the previous equality we
obtain,
\[D^{\dag}e=\frac{1}{\norm{\tau}^{2}}L\tau-\frac{\tau^{t}L\tau}{\norm{\tau}^{4}}\tau.\]
Replacing this expression for $D^{\dag}e$ in (\ref{E MPI D wtree
u}), we get (\ref{E MPI D wtree u e}).
\end{proof}

\subsection{Moore--Penrose inverses of distance matrices of wheel
graphs with an odd number of vertices}\label{SubS MPI dmwgo}

Let $n \geq 5$ be an odd integer. Let $W(n)$ be the wheel graph
having $n$ number of vertices, with the center labeled $1$, the
other vertices labeled $2$, ..., $n$ lie in a cycle of length $n-1$,
and $(i, i+1)$ is an edge. Without loss of generality, we fix this
labeling because any other labeling of $W(n)$ leads to a distance
matrix which is a permutation similar to $D$. If $u=(0, 1, 2,
\ldots, 2, 1) \in \mathbb{R}^{n-1}$, then the distance matrix $D$ of
the wheel graph is given by
\[D=\left(
     \begin{array}{cc}
       0 & e^{t} \\
       e & {\rm circ}(u) \\
     \end{array}
   \right).\]
The following proposition is about ${\rm null}(D)$ and $D^{\dag}$.
\begin{proposition}\label{P MPI dmwgo null space}
Let $n \geq 5$ odd. Let $D$ be the distance matrix of the wheel
graph with $n$ vertices and let $a=(0, -1, 1, \ldots, -1 , 1) \in
\mathbb{R}^{n}$. Then ${\rm null}(D)={\rm span}\{a\}$ and
\begin{equation}\label{E MPI dmwto D aat}
D^{\dag}=\paren{D+aa^{t}}^{-1}-\frac{1}{(n-1)^{2}}aa^{t}.
\end{equation}
\end{proposition}
\begin{proof}
Clearly, ${\rm span}\{a\} \subseteq {\rm null}(D)$. Let $x \in {\rm
null}(D)$. Since $x^{t}D(1,:)=0$, $\sum_{j=2}^{n}x(j)=0$. We have
$\sum_{i=1}^{n}D(i,:)=(n-1, 3+2(n-3), \ldots, 3+2(n-3))$. Since
$x^{t}D(i,:)=0$ for $i=2, \ldots, n$,
\begin{equation}\label{E T MPI dmwgo 1}
x(2)=-\frac{x(n)+x(3)}{2},
\end{equation}
\begin{equation}\label{E T MPI dmwgo 2}
x(j)=-\frac{x(j-1)+x(j+1)}{2}, \text{ for $j=3, \ldots, n-1$},
\end{equation}
\begin{equation}\label{E T MPI dmwgo 3}
x(n)=-\frac{x(n-1)+x(2)}{2},
\end{equation}
\noindent and

\centerline{$0=x^{t}\paren{\sum_{i=1}^{n}D(i,:)}=\paren{n-1}x(1)+\paren{3+2(n-3)}\sum_{j=2}^{n}x(j)=\paren{n-1}x(1)$.}

\noindent From the last equality, $x(1)=0$, and from (\ref{E T MPI
dmwgo 1})-(\ref{E T MPI dmwgo 2}),
$x(j)=\paren{-1}^{j}\paren{\paren{j-1}x(2)+\paren{j-2}x(n)}$ for
$j=3, \ldots, n$. In particular,
$x(n)=-\paren{\paren{n-1}x(2)+\paren{n-2}x(n)}$. Therefore,
$x(n)=-x(2)$ and, more generally, $x(j)=\paren{-1}^{j}x(2)$ for
$j=3, \ldots, n$. This shows that ${\rm null}(D)={\rm span}\{a\}$.
Now, from Corollary~\ref{C MPI expression full rank}(iii), we obtain
(\ref{E MPI dmwto D aat}).
\end{proof}

In \cite{Balaji-Bapat-Goel (2020)}, the equality
\begin{equation}\label{E MPI dmwto BBG}
D^{\dag}=-\frac{1}{2}\widetilde{L}+\frac{4}{n-1}ww^{t}
\end{equation}
where $\widetilde{L}$ is a semidefinite positive matrix, ${\rm
rank}\!\paren{\widetilde{L}}=n-2$, $\widetilde{L}e=0$ and
$w=\frac{1}{4}\paren{5-n,1,\ldots,1}$, is obtained. The matrix
$\widetilde{L}$ can be viewed as a special Laplacian matrix.
Equating (\ref{E MPI dmwto D aat}) and (\ref{E MPI dmwto BBG}) we
get
\begin{equation}\label{E MPI dmwto = BBG}
\paren{D+aa^{t}}^{-1}=-\frac{1}{2}\widetilde{L}+\frac{4}{n-1}ww^{t}+\frac{1}{(n-1)^{2}}aa^{t}.
\end{equation}
One expression of $\paren{D+aa^{t}}^{-1}$ can be derived from
(\ref{E MPI dmwto = BBG}) and the expression of $\widetilde{L}$
given in \cite[Definition 1]{Balaji-Bapat-Goel (2020)}. Next, we
give a closed-form expression of each entry of
$\paren{D+aa^{t}}^{-1}$ and hence, of each entry of $D^{\dag}$,
based on numerical experiments. The principal result is
Theorem~\ref{T MPI dmwgo}. We establish before three auxiliary
lemmas. The first of them is about sums. The other two lemmas are
about the entries of $\paren{D+aa^{t}}^{-1}$ and describe properties
of them.

\begin{lemma}\label{L MPI dmwgo sums}
Let $n \geq 5$ odd. The following equalities hold:
\begin{enumerate}
  \item[(i)] If $n=5+4m$ for some $m \in \mathbb{N}$, $m \geq 0$, then \[\sum_{k=1, k\, {\rm
  even}}^{(n-3)/2}k=\frac{\paren{n-5}\paren{n-1}}{16}=\frac{n^{2}-6n+5}{16},\]
\[\sum_{k=1, k\, {\rm
odd}}^{(n-3)/2}k=\frac{\paren{n-1}^{2}}{16}=\frac{n^{2}-2n+1}{16},\]
\[\sum_{k=1, k\, {\rm
even}}^{(n-3)/2}k^{2}=\frac{\paren{n-5}\paren{n-1}\paren{n-3}}{48}=\frac{n^{3}-9n^{2}+23n-15}{48}\]
and
\[\sum_{k=1, k\, {\rm
odd}}^{(n-3)/2}k^{2}=\frac{\paren{n-1}\paren{n-3}\paren{n+1}}{48}=\frac{n^{3}-3n^{2}-n+3}{48}.\]
  \item[(ii)] If $n=7+4m$ for some $m \in \mathbb{N}$, $m \geq 0$, then \[\sum_{k=2, k\, {\rm
  even}}^{(n-3)/2}k=\frac{\paren{n-3}\paren{n+1}}{16}=\frac{n^{2}-2n-3}{16},\] \[\sum_{k=1, k\, {\rm
odd}}^{(n-3)/2}k=\frac{\paren{n-3}^{2}}{16}=\frac{n^{2}-6n+9}{16},\]
\[\sum_{k=1, k\, {\rm
even}}^{(n-3)/2}k^{2}=\frac{\paren{n-3}\paren{n+1}\paren{n-1}}{48}=\frac{n^{3}-3n^{2}-n+3}{48}\]
and
\[\sum_{k=1, k\, {\rm
odd}}^{(n-3)/2}k^{2}=\frac{\paren{n-3}\paren{n-1}\paren{n-5}}{48}=\frac{n^{3}-9n^{2}+23n-15}{48}.\]
\end{enumerate}
\end{lemma}
\begin{proof}
(i): Assume that $n=5+4m$ for some $m \in \mathbb{N}$, $m \geq 0$.
We have
\begin{align*}
\sum_{k=2, k\, {\rm
  even}}^{(n-3)/2}k&=\sum_{k=1, k\, {\rm
even}}^{2m+1}k=\sum_{k=1}^{m}2k=2\frac{m\paren{m+1}}{2}\\
&=\frac{n-5}{4}\paren{\frac{n-5}{4}+1}=\frac{\paren{n-5}\paren{n-1}}{16}=\frac{n^{2}-6n+5}{16},
\end{align*}
\begin{align*}
\sum_{k=1, k\, {\rm odd}}^{(n-3)/2}k&=\sum_{k=1, k\, {\rm
odd}}^{2m+1}k=\sum_{k=1}^{m+1}\paren{2k-1}=2\frac{\paren{m+1}\paren{m+2}}{2}-\paren{m+1}\\
&=\paren{\frac{n-5}{4}+1}\paren{\frac{n-5}{4}+2}-\paren{\frac{n-5}{4}+1}=\frac{\paren{n-1}^{2}}{16}=\frac{n^{2}-2n+1}{16},
\end{align*}
\begin{align*}
\sum_{k=2, k\, {\rm
  even}}^{(n-3)/2}k^{2}&=\sum_{k=1, k\, {\rm
even}}^{2m+1}k^{2}=\sum_{k=1}^{m}\paren{2k}^{2}=4\frac{m\paren{m+1}\paren{2m+1}}{6}\\
&=\frac{2}{3}\frac{n-5}{4}\paren{\frac{n-5}{4}+1}\paren{\frac{n-5}{2}+1}=\frac{\paren{n-5}\paren{n-1}\paren{n-3}}{48}=\frac{n^{3}-9n^{2}+23n-15}{48}
\end{align*} and
\begin{align*}
\sum_{k=1, k\, {\rm odd}}^{(n-3)/2}k^{2}&=\sum_{k=1, k\, {\rm
odd}}^{2m+1}k^{2}=\sum_{k=1}^{m+1}\paren{2k-1}^{2}=4\sum_{k=1}^{m+1}k^{2}-4\sum_{k=1}^{m+1}k+m+1\\
&=4\frac{\paren{m+1}\paren{m+2}\paren{2\paren{m+1}+1}}{6}-4\frac{\paren{m+1}\paren{m+2}}{2}+m+1\\
&=\paren{m+1}\paren{\paren{m+2}\frac{4}{3}m+1}=\paren{\frac{n-5}{4}+1}\paren{\paren{\frac{n-5}{4}+2}\frac{n-5}{3}+1}\\
&=\frac{\paren{n-1}\paren{n-3}\paren{n+1}}{48}=\frac{n^{3}-3n^{2}-n+3}{48}.
\end{align*}
(ii): Assume now that $n=7+4m$ for some $m \in \mathbb{N}$, $m \geq
0$. Then
\begin{align*}
\sum_{k=1, k\, {\rm even}}^{(n-3)/2}k&=\sum_{k=1, k\, {\rm
even}}^{2m+2}k=2\sum_{k=1}^{m+1}k=\paren{m+1}\paren{m+2}\\
&=\paren{\frac{n-7}{4}+1}\paren{\frac{n-7}{4}+2}=\frac{\paren{n-3}\paren{n+1}}{16}=\frac{n^{2}-2n-3}{16},
\end{align*}
\begin{align*}
\sum_{k=1, k\, {\rm odd}}^{(n-3)/2}k&=\sum_{k=1, k\, {\rm
odd}}^{2m+2}k=\sum_{k=1}^{m+1}\paren{2k-1}\\
&=2\frac{\paren{m+1}\paren{m+2}}{2}-\paren{m+1}=\paren{m+1}^{2}\\
&=\paren{\frac{n-7}{4}+1}^{2}=\frac{\paren{n-3}^{2}}{16}=\frac{n^{2}-6n+9}{16},
\end{align*}
\begin{align*}
\sum_{k=1, k\, {\rm even}}^{(n-3)/2}k^{2}&=\sum_{k=1, k\, {\rm
even}}^{2m+2}k^{2}=\sum_{k=1}^{m+1}\paren{2k}^{2}=4\sum_{k=1}^{m+1}k^{2}\\
&=4\frac{\paren{m+1}\paren{m+2}\paren{2\paren{m+1}+1}}{6}=\frac{2}{3}\paren{\frac{n-7}{4}+1}\paren{\frac{n-7}{4}+2}\paren{\frac{n-7}{2}+3}\\
&=\frac{\paren{n-3}\paren{n+1}\paren{n-1}}{48}=\frac{n^{3}-3n^{2}-n+3}{48}
\end{align*} and
\begin{align*}
\sum_{k=1, k\, {\rm odd}}^{(n-3)/2}k^{2}&=\sum_{k=1, k\, {\rm
odd}}^{2m+2}k^{2}=\sum_{k=1}^{m+1}\paren{2k-1}^{2}=4\sum_{k=1}^{m+1}k^{2}-4\sum_{k=1}^{m+1}k+m+1\\
&=4\frac{\paren{m+1}\paren{m+2}\paren{2\paren{m+1}+1}}{6}-4\frac{\paren{m+1}\paren{m+2}}{2}+m+1\\
&=\frac{n^{3}-9n^{2}+23n-15}{48}=\frac{\paren{n-3}\paren{n-1}\paren{n-5}}{48}.
\end{align*}
\end{proof}

The vector $z$ of the following two lemmas will appear in the
expression of $D^{\dag}$ given in Theorem~\ref{T MPI dmwgo} below.

\begin{lemma}\label{L MPI dmwgo z1}
Let $z=\paren{z(0), \ldots, z(n-2)}$ where
\[z(0)=\frac{-n^{3}+3n^{2}+n+9}{12},\]
\[z(k)=z(n-1-k)=\left\{
    \begin{array}{ll}
      \frac{-6\paren{n-1}k^{2}+6\paren{n-1}^{2}k-n^{3}+3n^{2}+n+9}{12}, & \hbox{$1 \leq k \leq \frac{n-3}{2}$, $k$ even,} \\
      \frac{6\paren{n-1}k^{2}-6\paren{n-1}^{2}k+n^{3}-3n^{2}+5n-15}{12}, & \hbox{$1 \leq k \leq \frac{n-3}{2}$, $k$ odd,}
    \end{array}
  \right.
\] and \[z((n-1)/2)=\left\{
                       \begin{array}{ll}
                         \frac{n^{3}-3n^{2}+11n+15}{24}, & \hbox{$n=5+4m$, $m=0, 1,
\ldots$,} \\
                         \frac{-n^{3}+3n^{2}+n-27}{24}, & \hbox{$n=7+4m$, $m=0, 1,
\ldots$.}
                       \end{array}
                     \right.
\] Then
\begin{equation}\label{E L MPI dmwgo z1 sum even n-3}
\sum_{k=1, k\, {\rm even}}^{(n-3)/2}z(k)=\left\{
                       \begin{array}{ll}
                         \frac{n^{3}-3n^{2}-n-45}{48}, & \hbox{$n=5+4m$, $m=0, 1,
\ldots$,} \\
                         \frac{2n^{3}-6n^{2}+10n-30}{48}, & \hbox{$n=7+4m$, $m=0, 1,
\ldots$,}
                       \end{array}
                     \right.
\end{equation}
\begin{equation}\label{E L MPI dmwgo z1 sum odd n-3}
\sum_{k=1, k\, {\rm odd}}^{(n-3)/2}z(k)=\left\{
                       \begin{array}{ll}
                         -\frac{n-1}{4}, & \hbox{$n=5+4m$, $m=0, 1,
\ldots$,} \\
                         \frac{n^{3}-3n^{2}-13n+39}{48}, & \hbox{$n=7+4m$, $m=0, 1,
\ldots$,}
                       \end{array}
                     \right.
\end{equation}
\begin{equation}\label{E L MPI dmwgo z1 sum even odd n-2}
\sum_{k=0, k\, {\rm even}}^{n-2}z(k)=-\sum_{k=0, k\, {\rm
odd}}^{n-2}z(k)=\frac{n-1}{2},
\end{equation}
\begin{equation}\label{E L MPI dmwgo z1 2sum even n-2 1 n-2}
2\sum_{l=2, l\,{\rm
even}}^{n-3}z(l)-z(1)-z(n-2)=\paren{n-1}\paren{n-2},
\end{equation}
\begin{equation}\label{E L MPI dmwgo z1 2sum odd n-2 0 n-3}
2\sum_{l=1, l\,{\rm odd}}^{n-4}z(l)-z(0)-z(n-3)=-\paren{n-1},
\end{equation}
\begin{equation}\label{E L MPI dmwgo z1 k k-1 k+1 even}
2z(k)+z(k-1)+z(k+1)=2\paren{n-1} \text{ for $2 \leq k \leq n-3$ and
$k$ even,}
\end{equation}
\begin{equation}\label{E L MPI dmwgo z1 2sum  k-1 k+1 even}
2\sum_{l=0, l\,{\rm even}, l \neq
k}^{n-2}z(l)-z(k-1)-z(k+1)=-\paren{n-1} \text{ for $2 \leq k \leq
n-3$ and $k$ even,}
\end{equation}
\begin{equation}\label{E L MPI dmwgo z1 k k-1 k+1 odd}
2z(k)+z(k-1)+z(k+1)=0 \text{ for $1 \leq k \leq n-3$ and $k$ odd,}
\end{equation} and
\begin{equation}\label{E L MPI dmwgo z1 2sum  k-1 k+1 odd}
2\sum_{l=0, l\,{\rm odd}, l \neq
k}^{n-2}z(l)-z(k-1)-z(k+1)=-\paren{n-1} \text{for $1 \leq k \leq
n-3$ and $k$ odd}.
\end{equation}
\end{lemma}
\begin{proof}
We first note that if $n=5+4m$ for some $m \geq 0$, then
$\frac{n-3}{2}$ is odd and $\frac{n-1}{2}$ is even, whereas if
$n=7+4m$ for some $m \geq 0$, then $\frac{n-3}{2}$ is even and
$\frac{n-1}{2}$ is odd.

Assume that $n=5+4m$ for some $m \geq 0$. By Lemma~\ref{L MPI dmwgo
sums}(i),
\begin{align}\label{E L MPI dmwgo z1 5 sum even n-3}
\sum_{k=1, k\, {\rm
even}}^{(n-3)/2}z(k)=&\frac{1}{12}\paren{-6\paren{n-1}\sum_{k=1, k\,
{\rm even}}^{(n-3)/2}k^{2}+6\paren{n^{2}-2n+1}\sum_{k=1,
k\, {\rm even}}^{(n-3)/2}k+\paren{-n^{3}+3n^{2}+n+9}\frac{n-5}{4}}\notag\\
=&\frac{1}{12}\paren{-6\paren{n-1}\frac{n^{3}-9n^{2}+23n-15}{48}+6\paren{n^{2}-2n+1}\frac{n^{2}-6n+5}{16}}\notag\\
&+\frac{1}{12}\paren{\paren{-n^{3}+3n^{2}+n+9}\frac{n-5}{4}}\notag\\
=&\frac{n^{3}-3n^{2}-n-45}{48}
\end{align}
and
\begin{align}\label{E L MPI dmwgo z1 5 sum odd n-3}
\sum_{k=1, k\, {\rm
odd}}^{(n-3)/2}z(k)=&\frac{-1}{12}\paren{-6\paren{n-1}\sum_{k=1, k\,
{\rm odd}}^{(n-3)/2}k^{2}+6\paren{n^{2}-2n+1}\sum_{k=1, k\, {\rm odd}}^{(n-3)/2}k+\paren{-n^{3}+3n^{2}-5n+15}\frac{n-1}{4}}\notag\\
=&\frac{-1}{12}\paren{-6\paren{n-1}\frac{n^{3}-3n^{2}-n+3}{48}+6\paren{n^{2}-2n+1}\frac{n^{2}-2n+1}{16}}\notag\\
&+\frac{-1}{12}\paren{\paren{-n^{3}+3n^{2}-5n+15}\frac{n-1}{4}}\notag\\
=&-\frac{n-1}{4}.
\end{align}
Using (\ref{E L MPI dmwgo z1 5 sum even n-3}) and (\ref{E L MPI
dmwgo z1 5 sum odd n-3}) we get
\begin{align*}
\sum_{k=0, k\, {\rm even}}^{n-2}z(k)&=z(0)+\sum_{k=1, k\, {\rm
even}}^{(n-3)/2}z(k)+z((n-1)/2)+\sum_{k=(n+1)/2, k\, {\rm
even}}^{n-2}z(k)\notag\\
&=\frac{-n^{3}+3n^{2}+n+9}{12}+\frac{n^{3}-3n^{2}-n-45}{24}+\frac{n^{3}-3n^{2}+11n+15}{24}\notag\\
&=\frac{n-1}{2}
\end{align*}
and
\begin{align*}
\sum_{k=0, k\, {\rm odd}}^{n-2}z(k)&=\sum_{k=1, k\, {\rm
odd}}^{(n-3)/2}z(k)+\sum_{k=(n+1)/2, k\, {\rm
odd}}^{n-2}z(k)=-\frac{n-1}{2}.
\end{align*}
Assume now that $n=7+4m$ for some $m \geq 0$. By Lemma~\ref{L MPI
dmwgo sums}(ii),
\begin{align}\label{E L MPI dmwgo z1 7 sum even n-3}
\sum_{k=1, k\, {\rm
even}}^{(n-3)/2}z(k)=&\frac{1}{12}\paren{-6\paren{n-1}\sum_{k=1, k\,
{\rm even}}^{(n-3)/2}k^{2}+6\paren{n^{2}-2n+1}\sum_{k=1,
k\, {\rm even}}^{(n-3)/2}k+\paren{-n^{3}+3n^{2}+n+9}\frac{n-3}{4}}\notag\\
=&\frac{1}{12}\paren{-6\paren{n-1}\frac{n^{3}-3n^{2}-n+3}{48}+6\paren{n^{2}-2n+1}\frac{n^{2}-2n-3}{16}}\notag\\
&+\frac{1}{12}\paren{\paren{-n^{3}+3n^{2}+n+9}\frac{n-3}{4}}\notag\\
=&\frac{2n^{3}-6n^{2}+10n-30}{48}
\end{align}
and
\begin{align}\label{E L MPI dmwgo z1 7 sum odd n-3}
\sum_{k=1, k\, {\rm
odd}}^{(n-3)/2}z(k)=&\frac{-1}{12}\paren{-6\paren{n-1}\sum_{k=1, k\,
{\rm odd}}^{(n-3)/2}k^{2}+6\paren{n^{2}-2n+1}\sum_{k=1, k\, {\rm odd}}^{(n-3)/2}k+\paren{-n^{3}+3n^{2}-5n+15}\frac{n-3}{4}}\notag\\
=&\frac{1}{12}\paren{-6\paren{n-1}\frac{n^{3}-9n^{2}+23n-15}{48}+6\paren{n^{2}-2n+1}\frac{n^{2}-6n+9}{16}}\notag\\
&+\frac{1}{12}\paren{\paren{-n^{3}+3n^{2}-5n+15}\frac{n-3}{4}}\notag\\
=&\frac{n^{3}-3n^{2}-13n+39}{48}.
\end{align}
From (\ref{E L MPI dmwgo z1 7 sum even n-3}) and (\ref{E L MPI dmwgo
z1 7 sum odd n-3}) we obtain,
\begin{align*}
\sum_{k=0, k\, {\rm even}}^{n-2}z(k)&=z(0)+\sum_{k=1, k\, {\rm
even}}^{(n-3)/2}z(k)+\sum_{k=(n+1)/2, k\, {\rm
even}}^{n-2}z(k)\notag\\
&=\frac{-n^{3}+3n^{2}+n+9}{12}+\frac{2n^{3}-6n^{2}+10n-30}{24}=\frac{n-1}{2}
\end{align*}
and
\begin{align*}
\sum_{k=0, k\, {\rm odd}}^{n-2}z(k)&=\sum_{k=1, k\, {\rm
odd}}^{(n-3)/2}z(k)+z((n-1)/2)+\sum_{k=(n+1)/2, k\, {\rm
odd}}^{n-2}z(k)\notag\\
&=\frac{n^{3}-3n^{2}-13n+39}{24}+\frac{-n^{3}+3n^{2}+n-27}{24}=-\frac{n-1}{2}.
\end{align*}
Hence, we have proved (\ref{E L MPI dmwgo z1 sum even n-3}), (\ref{E
L MPI dmwgo z1 sum odd n-3}) and (\ref{E L MPI dmwgo z1 sum even odd
n-2}).

For each $n \geq 5$ odd, using (\ref{E L MPI dmwgo z1 sum even odd
n-2}) we obtain (\ref{E L MPI dmwgo z1 2sum even n-2 1 n-2}) and
(\ref{E L MPI dmwgo z1 2sum odd n-2 0 n-3})
\begin{align*}
2\sum_{l=2, l\,{\rm even}}^{n-3}z(l)-z(1)-z(n-2)&=2\sum_{l=0,
l\,{\rm even}}^{n-3}z(l)-2z(0)-z(1)-z(n-2)\\
&=2\sum_{l=0,
l\,{\rm even}}^{n-3}z(l)-2z(0)-2z(1)\\
&=n-1-\frac{-n^{3}+3n^{2}+n+9}{6}-\frac{n^{3}-9n^{2}+23n-27}{6}\\
&=\paren{n-1}\paren{n-2}
\end{align*}
and
\begin{align*}
2\sum_{l=1, l\,{\rm odd}}^{n-4}z(l)-z(0)-z(n-3)=&2\sum_{l=1, l\,{\rm odd}}^{n-2}z(l)-2z(n-2)-z(0)-z(n-3)\\
&=2\sum_{l=1, l\,{\rm odd}}^{n-2}z(l)-2z(1)-z(0)-z(2)\\
=&1-n-2\frac{6\paren{n-1}-6\paren{n-1}^{2}+n^{3}-3n^{2}+5n-15}{12}\\
&-\frac{-n^{3}+3n^{2}+n+9}{12}-\frac{-24\paren{n-1}+12\paren{n-1}^{2}-n^{3}+3n^{2}+n+9}{12}\\
&=-\paren{n-1}.
\end{align*}

Assume that $2 \leq k \leq n-3$ and $k$ even. We note that $k-1$,
$k+1$, $n-k-2$, and $n-k$ are odd and $n-k-1$ is even. If $k+1 <
\frac{n-1}{2}$, then
\begin{align*}
2z(k)+z(k-1)+z(k+1)=&2\frac{-6\paren{n-1}k^{2}+6\paren{n-1}^{2}k-n^{3}+3n^{2}+n+9}{12}\\
&+\frac{6\paren{n-1}\paren{k-1}^{2}-6\paren{n-1}^{2}\paren{k-1}+n^{3}-3n^{2}+5n-15}{12}\\
&+\frac{6\paren{n-1}\paren{k+1}^{2}-6\paren{n-1}^{2}\paren{k+1}+n^{3}-3n^{2}+5n-15}{12}\\
&=2\paren{n-1}.
\end{align*}
If $k-1 > \frac{n-1}{2}$, then
\begin{align*}
2z(k)+z(k-1)+z(k+1)=&2z(n-k-1)+z(n-k)+z(n-k-2)\\
=&\frac{-12\paren{n-1}\paren{n-k-1}^{2}+12\paren{n-1}^{2}\paren{n-k-1}-2n^{3}+6n^{2}+2n+18}{12}\\
&+\frac{6\paren{n-1}\paren{n-k}^{2}-6\paren{n-1}^{2}\paren{n-k}+n^{3}-3n^{2}+5n-15}{12}\\
&+\frac{6\paren{n-1}\paren{n-k-2}^{2}-6\paren{n-1}^{2}\paren{n-k-2}+n^{3}-3n^{2}+5n-15}{12}\\
=&2\paren{n-1}.
\end{align*}
If $n=5+4m$ for some $m \geq 0$ and $k = \frac{n-1}{2}$, then
\begin{align*}
2z(k)+z(k-1)+z(k+1)=&2z((n-1)/2)+z((n-3)/2)+z((n-3)/2+1)\\
=&2\paren{z((n-1)/2)+z((n-3)/2)}\\
=&\frac{n^{3}-3n^{2}+11n+15}{12}\\
&+\frac{3\paren{n-1}\paren{n-3}^{2}-6\paren{n-1}^{2}\paren{n-3}+2n^{3}-6n^{2}+10n-30}{12}\\
=&2\paren{n-1}.
\end{align*}
If $n=7+4m$ for some $m \geq 0$ and $k+1=\frac{n-1}{2}$, then
\begin{align*}
2z(k)+z(k-1)+z(k+1)=&2z((n-3)/2)+z((n-5)/2)+z((n-1)/2)\\
=&2\frac{-6\paren{n-1}\paren{\frac{n-3}{2}}^{2}+6\paren{n-1}^{2}\frac{n-3}{2}-n^{3}+3n^{2}+n+9}{12}\\
&+\frac{6\paren{n-1}\paren{\frac{n-5}{2}}^{2}-6\paren{n-1}^{2}\frac{n-5}{2}+n^{3}-3n^{2}+5n-15}{12}\\
&+\frac{-n^{3}+3n^{2}+n-27}{24}\\
=&2\paren{n-1}.
\end{align*}
If $n=7+4m$ for some $m \geq 0$ and $k-1=\frac{n-1}{2}$, using the
above equality we get
\begin{align*}
2z(k)+z(k-1)+z(k+1)=&2z((n+1)/2)+z((n-1)/2)+z((n+3)/2)\\
=&2z((n-3)/2)+z((n-1)/2)+z((n-5)/2)=2\paren{n-1}.
\end{align*}
Thus, (\ref{E L MPI dmwgo z1 k k-1 k+1 even}) is proved. Using
(\ref{E L MPI dmwgo z1 sum even odd n-2}) and (\ref{E L MPI dmwgo z1
k k-1 k+1 even}) we obtain (\ref{E L MPI dmwgo z1 2sum  k-1 k+1
even}):
\begin{align*}
2\sum_{l=0, l\,{\rm even}, l \neq
k}^{n-2}z(l)-z(k-1)-z(k+1)&=2\sum_{l=0, l\,{\rm
even}}^{n-2}z(l)-2z(k)-z(k-1)-z(k+1)\\
&=n-1-2\paren{n-1}=-\paren{n-1}.
\end{align*}

Assume that $1 \leq k \leq n-3$, and $k$ odd. We note that $k-1$,
$k+1$, $n-k-2$, and $n-k$ are even, and $n-k-1$ is odd. If $k+1 <
\frac{n-1}{2}$, then
\begin{align*}
2z(k)+z(k-1)+z(k+1)=&2\frac{6\paren{n-1}k^{2}-6\paren{n-1}^{2}k+n^{3}-3n^{2}+5n-15}{12}\\
&+\frac{-6\paren{n-1}\paren{k-1}^{2}+6\paren{n-1}^{2}\paren{k-1}-n^{3}+3n^{2}+n+9}{12}\\
&+\frac{-6\paren{n-1}\paren{k+1}^{2}+6\paren{n-1}^{2}\paren{k+1}-n^{3}+3n^{2}+n+9}{12}\\
=&0.
\end{align*}
If $k-1 > \frac{n-1}{2}$, then
\begin{align*}
2z(k)+z(k-1)+z(k+1)=&2z(n-k-1)+z(n-k)+z(n-k-2)\\
=&2\frac{6\paren{n-1}\paren{n-k-1}^{2}-6\paren{n-1}^{2}\paren{n-k-1}+n^{3}-3n^{2}+5n-15}{12}\\
&+\frac{-6\paren{n-1}\paren{n-k}^{2}+6\paren{n-1}^{2}\paren{n-k}-n^{3}+3n^{2}+n+9}{12}\\
&+\frac{-6\paren{n-1}\paren{n-k-2}^{2}+6\paren{n-1}^{2}\paren{n-k-2}-n^{3}+3n^{2}+n+9}{12}\\
=&0.
\end{align*}
If $n=5+4m$ for some $m \geq 0$ and $k+1=\frac{n-1}{2}$, then
\begin{align*}
2z(k)+z(k-1)+z(k+1)=&2z((n-3)/2)+z((n-5)/2)+z((n-1)/2)\\
=&2\frac{6\paren{n-1}\paren{\frac{n-3}{2}}^{2}-6\paren{n-1}^{2}\paren{\frac{n-3}{2}}+n^{3}-3n^{2}+5n-15}{12}\\
&+\frac{-6\paren{n-1}\paren{\frac{n-5}{2}}^{2}+6\paren{n-1}^{2}\paren{\frac{n-5}{2}}-n^{3}+3n^{2}+n+9}{12}\\
&+\frac{n^{3}-3n^{2}+11n+15}{24}\\
=&0.
\end{align*}
If $n=5+4m$ for some $m \geq 0$ and $k-1=\frac{n-1}{2}$, using the
previous equality
\begin{align*}
2z(k)+z(k-1)+z(k+1)=&2z((n-2)/2+1)+z((n-1)/2)+z((n-1)/2+2)\\
=&2z(n-(n-1)/2-2)+z((n-1)/2)+z(n-(n-1)/2-3)\\
=&2z((n-3)/2)+z((n-1)/2)+z((n-5)/2)=0.
\end{align*}
If $n=7+4m$ for some $m \geq 0$ and $k=\frac{n-1}{2}$, then
\begin{align*}
2z(k)+z(k-1)+z(k+1)=&2z((n-1)/2)+z((n-3)/2)+z((n-2)/2+1)\\
=&2z((n-1)/2)+2z((n-3)/2)\\
=&2\frac{-n^{3}+3n^{2}+n-27}{24}\\
&+2\frac{-6\paren{n-1}\paren{\frac{n-3}{2}}^{2}+6\paren{n-1}^{2}\paren{\frac{n-3}{2}}-n^{3}+3n^{2}+n+9}{12}\\
=&0.
\end{align*}
This proves (\ref{E L MPI dmwgo z1 k k-1 k+1 odd}). From (\ref{E L
MPI dmwgo z1 sum even odd n-2}) and (\ref{E L MPI dmwgo z1 k k-1 k+1
odd}) we get (\ref{E L MPI dmwgo z1 2sum  k-1 k+1 odd}):
\begin{align*}
2\sum_{l=0, l\,{\rm odd}, l \neq
k}^{n-2}z(l)-z(k-1)-z(k+1)&=2\sum_{l=0, l\,{\rm
odd}}^{n-2}z(l)-2z(k)-z(k-1)-z(k+1)\\
&=2\paren{-\frac{n-1}{2}}-0=-\paren{n-1}.
\end{align*}
\end{proof}

\begin{lemma}\label{L MPI dmwgo z2}
Let $z$ as in Lemma~\ref{L MPI dmwgo z1}. Then \[e^{t}{\rm
circ}(z)=0\] and \[{\rm circ}(1, 0, 3, 1, \ldots, 3, 1, 3, 0){\rm
circ}(z)=\paren{n-1}^{2}I-\paren{n-1}ee^{t}.\]
\end{lemma}
\begin{proof}
Using Lemma~\ref{L MPI dmwgo z1}, \[e^{t}{\rm
circ}(z)=\sum_{k=0}^{n-2}z(k)=\sum_{k=0, k\, {\rm
even}}^{n-2}z(k)+\sum_{k=0, k\, {\rm
odd}}^{n-2}z(k)=\frac{n-1}{2}-\frac{n-1}{2}=0.\] For the other
equality we have,
\begin{align*}
{\rm circ}&(1, 0, 3, 1, \ldots, 3, 1, 3, 0){\rm
circ}(z)=\\
=&\paren{\Pi_{n-1}^{0}+3\paren{\Pi_{n-1}^{2}+\Pi_{n-1}^{4}+\ldots+\Pi_{n-1}^{n-5}+\Pi_{n-1}^{n-3}}+\paren{\Pi_{n-1}^{3}+\Pi_{n-1}^{5}+\ldots+\Pi_{n-1}^{n-6}+\Pi_{n-1}^{n-4}}}\\
&\paren{\sum_{k=0}^{n-2}z(k)\Pi_{n-1}^{k}}\\
=&\paren{ee^{t}+2\paren{\Pi_{n-1}^{2}+\Pi_{n-1}^{4}+\ldots+\Pi_{n-1}^{n-5}+\Pi_{n-1}^{n-3}}-\Pi_{n-1}^{1}-\Pi_{n-1}^{n-2}}\paren{\sum_{k=0}^{n-2}z(k)\Pi_{n-1}^{k}}\\
=&2\sum_{k=0}^{n-2}z(k)\sum_{k'=2, k', {\rm
even}}^{n-3}\Pi_{n-1}^{k+k'}-\sum_{k=0}^{n-2}z(k)\Pi_{n-1}^{k+1}-\sum_{k=0}^{n-2}z(k)\Pi_{n-1}^{n+k-2}\\
=&\paren{2\sum_{l=2, l\,{\rm
even}}^{n-3}z(l)}\Pi_{n-1}^{0}+\sum_{k=2, k\,{\rm
even}}^{n-3}\paren{2\sum_{l=0, l\,{\rm even}, l \neq
k}^{n-2}z(l)}\Pi_{n-1}^{k}\\
&+\sum_{k=1, k\,{\rm odd}}^{n-3}\paren{2\sum_{l=0, l\,{\rm odd}, l
\neq k}^{n-2}z(l)}\Pi_{n-1}^{k}+\paren{2\sum_{l=1,
l\,{\rm odd}}^{n-4}z(l)}\Pi_{n-1}^{n-2}\\
&-\paren{z(n-2)\Pi_{n-1}^{0}+\sum_{k=1}^{n-2}z(k-1)\Pi_{n-1}^{k}}-\paren{\sum_{k=0}^{n-3}z(k+1)\Pi_{n-1}^{k}+z(0)\Pi_{n-1}^{n-2}}\\
=&\paren{2\sum_{l=2, l\,{\rm
even}}^{n-3}z(l)-z(1)-z(n-2)}\Pi_{n-1}^{0}\\
&+\sum_{k=2, k\,{\rm even}}^{n-3}\paren{2\sum_{l=0, l\,{\rm even}, l
\neq
k}^{n-2}z(l)-z(k-1)-z(k+1)}\Pi_{n-1}^{k}\\
&+\sum_{k=1, k\,{\rm odd}}^{n-3}\paren{2\sum_{l=0, l\,{\rm odd}, l
\neq k}^{n-2}z(l)-z(k-1)-z(k+1)}\Pi_{n-1}^{k}\\
&+\paren{2\sum_{l=1, l\,{\rm
odd}}^{n-4}z(l)-z(0)-z(n-3)}\Pi_{n-1}^{n-2}.
\end{align*}
Now, applying Lemma~\ref{L MPI dmwgo z1}, we obtain \[{\rm circ}(1,
0, 3, 1, \ldots, 3, 1, 3, 0){\rm
circ}(z)=\paren{n-1}\paren{n-2}I-\paren{n-1}\sum_{k=1}^{n-2}\Pi_{n-1}^{k}=\paren{n-1}^{2}I-\paren{n-1}ee^{t}.\]
\end{proof}

Now we are in position to give our expression for $D^{\dag}$ using
the vector $z$ of Lemma~\ref{L MPI dmwgo z1}.

\begin{theorem}\label{T MPI dmwgo}
Let $n \geq 5$ odd. Let $D$ be the distance matrix of the wheel
graph with $n$ vertices. Let $a=(0, -1, 1, \ldots, -1 , 1) \in
\mathbb{R}^{n}$ and $z$ as in Lemma~\ref{L MPI dmwgo z1}. Let $v=(1,
-1, \ldots, 1, -1) \in \mathbb{R}^{n-1}$. Then
\begin{equation}\label{E T MPI dmwgo D+aat inv}
   \paren{D+aa^{t}}^{-1}=\frac{1}{\paren{n-1}^{2}}\left(
     \begin{array}{cc}
       -2\paren{n-1}\paren{n-3} & (n-1)e^{t} \\
       (n-1)e & {\rm circ}(z) \\
     \end{array}
   \right)
\end{equation}
and
\begin{equation}\label{E MPI dmwgo expression with A z}
D^{\dag}=\frac{1}{\paren{n-1}^{2}}\left(
     \begin{array}{cc}
       -2\paren{n-1}\paren{n-3} & (n-1)e^{t} \\
       (n-1)e & {\rm circ}(z+v) \\
     \end{array}
   \right).
\end{equation}
\end{theorem}
\begin{proof}
We have
\begin{equation}\label{E T MPI dmwgo aat}
aa^{t}=\left(
     \begin{array}{cc}
       0 & 0 \\
       0 & {\rm circ}(v) \\
     \end{array}
   \right)
\end{equation} and \[D+aa^{t}=\left(
     \begin{array}{cc}
       0 & e^{t} \\
       e & {\rm circ}(u+v)\\
     \end{array}
   \right)=\left(
     \begin{array}{cc}
       0 & e^{t} \\
       e & {\rm circ}(1, 0, 3, 1,
\ldots, 3, 1, 3, 0) \\
     \end{array}
   \right).\] Let \[X=\left(
     \begin{array}{cc}
       -2\paren{n-1}\paren{n-3} & (n-1)e^{t} \\
       (n-1)e & {\rm circ}(z) \\
     \end{array}
   \right).\] Since
\[\paren{D+aa^{t}}X=\left(
     \begin{array}{cc}
       \paren{n-1}^{2} & e^{t}{\rm circ}(z) \\
       0 & (n-1)ee^{t}+{\rm circ}(1, 0, 3, 1, \ldots, 3, 1, 3, 0){\rm
circ}(z) \\
     \end{array}
   \right),\]
applying Lemma~\ref{L MPI dmwgo z2}, we get
$\paren{D+aa^{t}}X=\paren{n-1}^{2}I$. This shows (\ref{E T MPI dmwgo
D+aat inv}). By Proposition~\ref{P MPI dmwgo null space}, (\ref{E T
MPI dmwgo D+aat inv}) and (\ref{E T MPI dmwgo aat}), we obtain
(\ref{E MPI dmwgo expression with A z}).
\end{proof}

We end this section with properties of $\paren{D+aa^{t}}^{-1}$,
$D^{\dag}$, and the matrix $\widetilde{L}$ of (\ref{E MPI dmwto
BBG}).

\begin{proposition}\label{P MPI dmwgo properties 123D}
Let $n \geq 5$ odd. Let $D$ be the distance matrix of the wheel
graph with $n$ vertices. Let $a=(0, -1, 1, \ldots, -1 , 1) \in
\mathbb{R}^{n}$. Let $w=\frac{1}{4}\paren{5-n,1,\ldots,1}$ and
$\widetilde{L}$ be such that
\[D^{\dag}=-\frac{1}{2}\widetilde{L}+\frac{4}{n-1}ww^{t}.\] Then
\begin{equation}\label{E MPI dmwto pa}
\paren{D+aa^{t}}^{-1}a=\frac{1}{n-1}a,
\end{equation}
\begin{equation}\label{E MPI dmwto p}
D^{\dag}=\paren{D+aa^{t}}^{-1}\paren{I-\frac{1}{n-1}aa^{t}},
\end{equation}
\begin{equation}\label{E MPI dmwto Lmonio}
\widetilde{L}=-2\paren{\paren{D+aa^{t}}^{-1}-\frac{4}{n-1}ww^{t}}\paren{I-\frac{1}{n-1}aa^{t}}
\end{equation} and
$\paren{D+aa^{t}}^{-1}-\frac{4}{n-1}ww^{t}$ is negative semidefinite
on ${\rm R}\!\paren{D}$.
\end{proposition}
\begin{proof}
By the equality $\paren{D+aa^{t}}a=n-1a$, we get (\ref{E MPI dmwto
pa}). Equality (\ref{E MPI dmwto p}) is a consequence of (\ref{E MPI
dmwto pa}) and (\ref{E MPI dmwto D aat}). Using that $w^{t}a=0$,
from (\ref{E MPI dmwto BBG}) and (\ref{E MPI dmwto p}), (\ref{E MPI
dmwto Lmonio}) follows.

Since ${\rm N}\!\paren{D}={\rm span}\set{a}$, $P_{{\rm
R}\!\paren{D}}x=I-\frac{1}{n-1}aa^{t}$. By \cite[Theorem
2]{Balaji-Bapat-Goel (2020)}, $\widetilde{L}$ is positive
semidefinite. Hence, by (\ref{E MPI dmwto Lmonio}) and noting that
${\rm R}\!\paren{\widetilde{L}} \subseteq {\rm R}\!\paren{D}$,
\[x^{t}\widetilde{L}x=-2\paren{P_{{\rm
R}\!\paren{D}}x}^{t}\paren{\paren{D+aa^{t}}^{-1}-\frac{4}{n-1}ww^{t}}P_{{\rm
R}\!\paren{D}}x,\] and we conclude that
$\paren{D+aa^{t}}^{-1}-\frac{4}{n-1}ww^{t}$ is negative semidefinite
on ${\rm R}\!\paren{D}$.
\end{proof}

\bigskip
{\bf Acknowledgment.} The author thanks the reviewer for the
comments that improved the paper.


\end{document}